\documentclass[11pt,a4paper]{amsart}
\usepackage[english]{babel}

\newtheorem{theorem}{Theorem}[section]

\newtheorem{corollary}[theorem]{Corollary}
\newtheorem{lemma}[theorem]{Lemma}
\newtheorem{proposition}[theorem]{Proposition}

\theoremstyle{definition}
\newtheorem*{definition}{Definition}
\newtheorem*{terminology}{Terminology}
\newtheorem{example}[theorem]{Example}
\newtheorem{remark}[theorem]{Remark}

\usepackage[all]{xy}
\usepackage{pstricks}
\usepackage{enumerate}
\usepackage{amsfonts,amssymb,amsmath,eucal,pinlabel,array,hhline}
\usepackage{slashed}
\usepackage{tabulary}
\usepackage{fancyhdr}
\usepackage{a4wide}
\usepackage{calrsfs,bbm,dsfont}
\usepackage[position=b]{subcaption}
\usepackage[square,sort,comma,numbers]{natbib}

\newcommand{\Z}{\mathds{Z}}

\newcommand{\R}{\mathds{R}}
\newcommand{\C}{\mathds{C}}

\newcommand{\id}{\mathit{id}}
\newcommand{\sign}{\mathit{sign}}
\newcommand{\eps}{\varepsilon}
 
\DeclareSymbolFont{EulerScript}{U}{eus}{m}{n}
\DeclareSymbolFontAlphabet\mathscr{EulerScript}
\begin{document}

\title{Colored tangles and signatures}

\author{David Cimasoni}
\address{Universit\'e de Gen\`eve, Section de math\'ematiques, 2-4 rue du Li\`evre, 1211 Gen\`eve 4, Switzerland}
\email{david.cimasoni@unige.ch}
\author{Anthony Conway}
\address{Universit\'e de Gen\`eve, Section de math\'ematiques, 2-4 rue du Li\`evre, 1211 Gen\`eve 4, Switzerland}
\email{anthony.conway@unige.ch}


\dedicatory{This paper is dedicated to the memory of Ruty Ben-Zion}

\begin{abstract}
Taking the signature of the closure of a braid defines a map from the braid group to the integers. In 2005, Gambaudo and Ghys expressed the homomorphism defect of this map
in terms of the Meyer cocycle and the Burau representation. In the present paper, we simultaneously extend this result in two directions, considering the multivariable signature of the closure of a colored tangle. The corresponding defect is expressed in terms of the Maslov index and of the Lagrangian functor defined by Turaev and the first-named author.
\end{abstract}

\maketitle


\section{Introduction}
\label{sec:intro}

Consider an arbitrary link invariant~$\mathcal{I}$ taking values in an abelian group. Precomposing this invariant with the braid closure defines maps~$\alpha\mapsto\mathcal{I}(\widehat{\alpha})$ from the braid groups~$B_n$ to this abelian group, and one might wonder whether these maps are group homomorphisms. In other words, one can ask whether
\[
\mathcal{I}(\widehat{\alpha\beta})-\mathcal{I}(\widehat{\alpha})-\mathcal{I}(\widehat{\beta})
\]
vanishes for all~$\alpha,\beta\in B_n$. This question has an easy answer: the only invariant with this property is the trivial one.
However, one can ask the more refined question of ``evaluating'' the homomorphism defect displayed above.
This can yield interesting consequences, both from the theoretical viewpoint, if this defect is expressed in terms of {\em a priori\/}
unrelated objects, and from the practical viewpoint, as it reduces the computation of the invariant to the computation of this defect (together with the value of~$\mathcal{I}$ on the closure of the standard generators of the braid group, i.e. unlinks).

This program was carried out by Gambaudo and Ghys~(\cite{GG}, recently republished in~\cite{Ensaios}) in the case of the Levine-Tristram signature~\cite{Mur,Lev,Tri}. This classical invariant
associates to an oriented link~$L$ an integral-valued map
\[
\sign(L)\colon S^1\to\Z, \quad\omega\mapsto\sign_\omega(L)\,.
\]
The great success of Gambaudo and Ghys was to express the homomorphism defect of this signature in terms of another classical object, the (reduced) Burau representation~\cite{Bur,Bir}
\[
\mathcal{B}_t \colon B_n\to\mathit{GL}_{n-1}(\Z[t,t^{-1}])\,.
\]
More precisely, it is known that this representation is unitary with respect to some skew-Hermitian form~\cite{Squier}. Therefore, given two braids~$\alpha,\beta\in B_n$ and a root of unity~$\omega$, one can consider the {\em Meyer cocycle\/} of the two unitary matrices~$\mathcal{B}_{\omega}(\alpha)$ and~$\mathcal{B}_{\omega}(\beta)$. The main theorem of~\cite{GG} is the equality
\begin{equation}
\label{equ:GG}
\sign_\omega(\widehat{\alpha\beta})-\sign_\omega(\widehat{\alpha})-\sign_\omega(\widehat{\beta})=-\mathit{Meyer}(\mathcal{B}_\omega(\alpha),\mathcal{B}_\omega(\beta))
\end{equation}
for all~$\alpha,\beta\in B_n$ and~$\omega\in S^1$ of order coprime to~$n$. (These authors actually work with the braid group on infinitely many strands~$B_\infty$, and obtain an equality valid for any~$\omega$ of finite order; however, their proof does yield the finer result stated above.)
Let us mention that this equality not only
relates two very much studied objects in knot theory, but also gives a very efficient algorithm for the
computation of the signature, as the Meyer cocycle is easy to calculate (and the signature of unlinks vanishes).

\medskip

Recall that the Levine-Tristram signature admits a generalization, the so-called {\em multivariable signature\/}~\cite{Gil,Cooper,Flo,CF}, which associates to a~$\mu$-colored
link~$L$ a map
\[
\sign(L)\colon\mathbb{T}^\mu\to\Z, \quad\omega=(\omega_1,\dots,\omega_\mu)\mapsto\sign_\omega(L)
\]
on the~$\mu$-dimensional torus~$\mathbb{T}^\mu$ (see subsection~\ref{sub:sign} below). The Burau representation has a
multivariable extension as well, called the {\em Gassner representation\/}~\cite{Bir}, which is unitary~\cite{Abd}, and it is natural to wonder if~(\ref{equ:GG})
holds in this multivariable setting.

Also, braids are but a very special kind of {\em tangles\/}, whose precise definition will be given in subsection~\ref{sub:tangle}. Oriented tangles no longer form groups, but are the morphisms of a category. Furthermore, the tangles that are endomorphims of a given object of this category can not only be composed, but also closed up to give oriented links, just like braids. Therefore, it makes sense to ask the same question as above, i.e. try to evaluate the defect of additivity of the signature on
tangles. It turns out that the Burau representation admits an extension to tangles~\cite{CT}, which takes the form of a functor~$\mathcal{F}$ from the category of oriented tangles to some {\em Lagrangian category\/} over the ring~$\Z[t,t^{-1}]$. It extends the Burau representation in the sense that if the tangle is a braid~$\alpha$, then~$\mathcal{F}(\alpha)$ is the graph of the unitary automorphism~$\mathcal{B}_t(\alpha)$ (see subsection~\ref{sub:functor} below). One cannot consider the Meyer cocycle of (pairs of) objects in this Lagrangian category, but it makes sense to consider the {\em Maslov index\/} of three objects in this category, evaluated at some~$t=\omega\in S^1$ (see subsection~\ref{sub:MM}). Therefore, one can ask whether the additivity defect of the signature of tangles
is related to the Maslov index of the image by~$\mathcal{F}$ of these tangles, evaluated at~$t=\omega$.

In the present paper, we answer both these questions simultaneously. The precise statement will be given in Theorem~\ref{thm:main} and Theorem~\ref{thm:relating} below, but in a nutshell, it can be phrased as follows.

\begin{theorem}
Given an object~$c$ of the category of~$\mu$-colored tangles and two endomorphisms~$\tau_1,\tau_2$ of this object, the equality
\[
\sign_\omega(\widehat{\tau_1\tau_2})-\sign_\omega(\widehat{\tau_1})-\sign_\omega(\widehat{\tau_2})=\mathit{Maslov}(\mathcal{F}_\omega(\overline{\tau}_1),\mathcal{F}_\omega(\mathit{id}_c),\mathcal{F}_\omega(\tau_2))
\]
holds for all~$\omega=(\omega_1,\dots,\omega_\mu)$ in an open dense subset of~$\mathbb{T}^\mu$, where~$\overline\tau$
denotes the horizontal reflection of the tangle~$\tau$, and~$\mathcal{F}_\omega$ is the evaluation at~$t=\omega$ of the multivariable extension of the Lagrangian functor~$\mathcal{F}$.
\end{theorem}

In the case of colored braids, this functor gives back the graph of the Gassner representation, the horizontal reflection of a braid is its inverse, and the Maslov index of the graphs of unitary automorphisms~$\gamma^{-1}_1,\mathit{id}$ and~$\gamma_2$ is one possible definition of the Meyer cocycle of~$\gamma_1$ and~$\gamma_2$. Therefore, in the case of~$\mu$-colored braids, our theorem is exactly the expected multivariable extension of (\ref{equ:GG}). 

We would like to point out that, although our demonstration roughly follows the same lines as the original proof of Gambaudo and Ghys,
several clarifications are made along the way. Actually, the paper~\cite{GG} contains a very detailed proof
in the case~$\omega=-1$, but only a brief description of the necessary modifications needed for the case of~$\omega$ a root
of unity. Therefore, it is our hope that the present paper will be of use not only to those interested in the full generality of our main result
(Theorem~\ref{thm:main}), but also to those merely curious about oriented tangles (Corollary~\ref{cor:oriented}), colored braids (Corollary~\ref{cor:braid}),
or a new algorithm for the computation of multivariable signatures (Remark~\ref{rem:alg}).

\medskip

The article is organized as follows. Section~\ref{sec:def} contains the definitions of the main objects involved, as well as the
statement of our theorem, of corollaries, and several examples. The necessary preliminaries are gathered in Section~\ref{sec:prelim}. The proof of the main 
theorem is given in Section~\ref{sec:proof}. Finally, Section~\ref{sec:rel} relates the functor~$\mathcal{F}_\omega$ to the Lagrangian functor~$\mathcal{F}$.

\subsection*{Acknowledgments.} The authors wish to express their thanks to Stefan Friedl and Maxime Bourrigan for useful discussions. The second-named author also thanks J\'er\'emy Dubout. Part of this paper was written while the first-named author was invited by the {\em Universit\'e Pierre et Marie Curie\/}, whose hospitality is thankfully acknowledged. This work was supported by a grant of the Swiss National Science Foundation.


\section{Definitions, statement of the theorem, and examples}
\label{sec:def}

The first aim of this section is to give precise definitions of the objects involved in this article: colored tangles, multivariable signatures,
isotropic and Lagrangian categories, the Maslov index and Meyer cocycle, and the isotropic functor are introduced in subsections~\ref{sub:tangle} to~\ref{sub:functor}. Our main result is then stated in subsection~\ref{sub:thm}, where several corollaries and examples are also given.

\subsection{Colored tangles and colored braids}
\label{sub:tangle}

Let~$D^2$ be the closed unit disk in~$\R^2$. Given a positive integer~$n$, let~$x_j$ be the point~$((2j-n-1)/n,0)$
in~$D^2$, for~$j=1,\dots,n$. Let~$\eps$ and~$\eps'$ be sequences of~$\pm 1$'s of respective length~$n$ and~$n'$.
An~{\em $(\eps,\eps')$-tangle} is the pair consisting of the cylinder~$D^2\times [0,1]$ and an oriented piecewise linear~$1$-submanifold~$\tau$ whose oriented boundary is~$\sum_{j=1}^{n'}\eps'_j(x'_j,1)-\sum_{j=1}^{n}\eps_i(x_j,0)$.

An oriented tangle~$\tau$ is called~{\em $\mu$-colored\/} if each of its components is assigned an element in~$\{1,\dots,\mu\}$.
We shall call a~$\mu$-colored~$(\eps,\eps')$-tangle a {\em $(c,c')$-tangle\/}, where~$c$ and~$c'$ are the sequences
of~$\pm 1,\pm 2,\dots,\pm\mu$ induced by the orientation and coloring of the tangle.
Given such a sequence~$c$, we shall denote by~$\ell(c)$ the element in~$\Z^\mu$ whose~$i^{\mathit{th}}$ coordinate is equal
to~$\ell(c)_i=\sum_{j;c_j=\pm i}\eps_j$. Note that for a~$(c,c')$-tangle to exist, we must have~$\ell(c)=\ell(c')$.

Two~$(c,c')$-tangles~$(D^2\times [0,1],\tau_1)$ and~$(D^2\times[0,1],\tau_2)$ are {\em isotopic\/} if there exists an
auto-homeomorphism~$h$ of~$D^2\times [0,1]$, keeping~$D^2\times\{0,1\}$ fixed, such that~$h(\tau_1)=\tau_2$ and~$h\vert_{\tau_1}\colon\tau_1\simeq\tau_2$ is orientation and color-preserving. We shall denote by~$T_\mu(c,c')$ the set of isotopy classes of~$(c,c')$-tangles, and by~$id_c$ the isotopy class of the trivial~$(c,c)$-tangle~$(D^2,\{x_1,\dots,x_n\})\times [0,1]$. 

Given a~$(c,c')$-tangle~$\tau_1$ and a~$(c',c'')$-tangle~$\tau_2$, their {\em composition\/} is the~$(c,c'')$-tangle~$\tau_2\tau_1$ obtained by gluing the two cylinders along the disk corresponding to~$c'$ and shrinking the length of the resulting cylinder by a factor~$2$ (see Figure \ref{fig:tangle}). Clearly, the composition of tangles induces a
composition~$T_\mu(c,c')\times T_\mu(c',c'')\to T_\mu(c,c'')$ on the isotopy classes of~$\mu$-colored tangles.

\begin{figure}[tb]
\labellist\small\hair 2.5pt
\pinlabel {$\tau_1$} at 20 600
\pinlabel {$\tau_2$} at 20 745
\pinlabel {$\tau_2\tau_1$} at 390 670
\pinlabel {$c$} at 72 548
\pinlabel {$c'$} at 72 690
\pinlabel {$c'$} at 72 655
\pinlabel {$c''$} at 72 799
\pinlabel {$c$} at 330 618
\pinlabel {$c''$} at 330 723
\endlabellist
\centering
\includegraphics[width=0.6\textwidth]{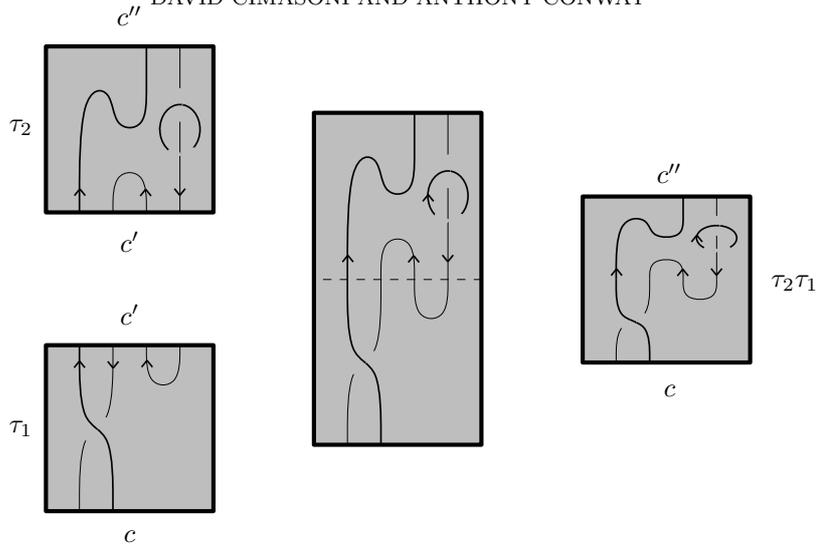}
\caption{A~$(c,c')$-tangle~$\tau_1$ with~$c=(-1,+2)$ and~$c'=(+2,-1,+1,-1)$, a~$(c',c'')$-tangle~$\tau_2$ with~$c''=(+2,-1)$, and their composition,
the~$(c,c'')$-tangle~$\tau_2\tau_1$.}
\label{fig:tangle}
\end{figure}

The category~$\mathbf{Tangles}_\mu$ of~$\mu$-colored tangles is defined as follows: the objects are the finite sequences~$c$ of elements in~$\{\pm 1,\pm 2,\dots,\pm\mu\}$, and the morphisms are given by~$\hbox{Hom}(c,c')=T_\mu(c,c')$. The composition is clearly associative, and the trivial tangle~$id_c$ plays the role of the identity endomorphism of~$c$.

We need to define some additional operations on tangles.
Given a~$(c,c')$-tangle~$\tau$, let us denote by~$\overline{\tau}$ the~$(c',c)$-tangle obtained from~$\tau$ by a reflection with respect to the
horizontal disk~$D^2\times\{1/2\}$.
Also, note that the category~$\mathbf{Tangles}_\mu$ is endowed with a monoidal structure given by the juxtaposition~$\tau_1\sqcup\tau_2$ of colored tangles. Finally, given an endomorphism~$\tau\in T_\mu(c,c)$, one can define its {\em closure\/} as the~$\mu$-colored link~$\widehat{\tau}$ obtained from~$\tau$ by adding oriented
colored parallel strands in~$S^3\setminus(D^2\times[0,1])$. All of these operations are illustrated in Figure~\ref{fig:tangle2}.

A~$(c,c')$-tangle~$\tau\subset D^2\times[0,1]$ is called a {\em colored braid\/} if every component is strictly increasing or strictly
decreasing with respect to the projection to~$[0,1]$. Finite sequences of elements in~$\{\pm 1,\pm 2,\dots,\pm\mu\}$, as objects,
and isotopy classes of colored braids, as morphisms, form a subcategory
\[
\mathbf{Braids}_\mu\subset\mathbf{Tangles}_\mu\,.
\]
In the case~$\mu=1$, these are simply the categories of oriented braids and oriented tangles.

For any sequence~$c$ of elements in~$\{\pm 1,\pm 2,\dots,\pm\mu\}$, the set~$B_c$ of endomorphisms
of~$c$ in~$\mathbf{Braids}_\mu$ is a group, with the inverse of~$\alpha\in B_c$ given by~$\overline{\alpha}=\alpha^{-1}$. If~$c=(1,\dots,1)$, then this group is nothing but the classical {\em braid group\/}~$B_n$.
If~$c=(1,2,\dots,n)$, then it is the {\em pure braid group\/}~$P_n$. For the purpose of this article, it is useful to interpolate
between and extend these two extreme cases, hence the natural notion of colored braids introduced in this paragraph.

\begin{figure}[tb]
\labellist\small\hair 2.5pt
\pinlabel {$\tau$} at 60 737
\pinlabel {$\overline{\tau}$} at 60 671
\pinlabel {$\tau_1$} at 210 737
\pinlabel {$\tau_2$} at 357 737
\pinlabel {$\tau_1\sqcup\tau_2$} at 280 671
\pinlabel {$\tau$} at 532 737
\pinlabel {$\widehat{\tau}$} at 590 613
\endlabellist
\includegraphics[width=0.9\textwidth]{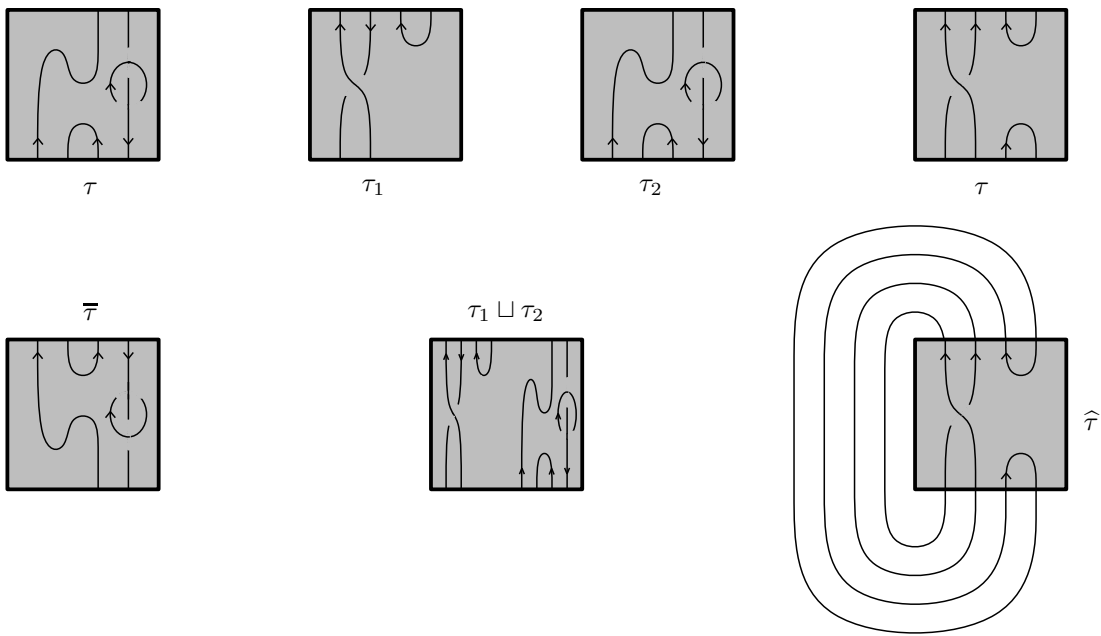}
\caption{From left to right: a tangle~$\tau$ and its reflection~$\overline{\tau}$, two tangles~$\tau_1,\tau_2$ and their juxtaposition~$\tau_1\sqcup\tau_2$, a tangle~$\tau$ and its closure~$\widehat{\tau}$.}
\label{fig:tangle2}
\end{figure}

\subsection{Multivariable signatures}
\label{sub:sign}

Recall that the {\em Levine-Tristram signature\/} of an oriented link~$L$ is the map
\[
\sign(L)\colon S^1\to\Z,\quad\omega\mapsto\sign_\omega(L)\,,
\]
where~$\sign_\omega(L)$ is the signature of the Hermitian matrix~$H(\omega)=(1-\omega)A+(1-\overline\omega)A^T$ and~$A$
is any Seifert matrix for the oriented link~$L$. In~\cite{CF}, this invariant was extended to arbitrary colored links; let us briefly recall this definition,
referring to~\cite{CF} for details, and to~\cite{Cooper,Gil,Flo} for previous related constructions.

A {\em~$C$-complex\/}~\cite{Cooper} for a~$\mu$-colored link~$L$ is a union~$S=S_1\cup\dots\cup S_\mu$ of surfaces in~$S^3$ such that:
\begin{enumerate}[(i)]
\item{for all~$i$,~$S_i$ is a  Seifert surface for the sublink of~$L$ of color~$i$;}
\item{for all~$i\neq j$,~$S_i\cap S_j$ is either empty or a union of clasps (see Figure \ref{fig:clasp});}
\item{for all~$i,j,k$ pairwise distinct,~$S_i\cap S_j\cap S_k$ is empty.}
\end{enumerate}

\begin{figure}[htb]
\labellist\small\hair 2.5pt
\pinlabel {$x$} at 18 58
\pinlabel {$S_i$} at 94 110
\pinlabel {$S_j$} at 331 118
\endlabellist
\centering
\includegraphics[width=0.4\textwidth]{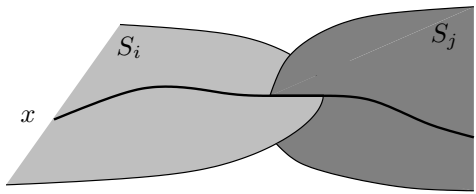}
\caption{A clasp intersection crossed by a~$1$-cycle~$x$.}
\label{fig:clasp}
\end{figure}

The existence of a~$C$-complex for an arbitrary colored link is fairly easy to establish, see~\cite[Lemma 1]{CimConway}.
Note that in the case~$\mu=1$, a~$C$-complex for~$L$ is nothing but a Seifert surface for the link~$L$.
Let us now define the corresponding generalization of the Seifert form.

Given a sequence~$\eps=(\eps_1,\dots,\eps_\mu)$ of~$\pm 1$'s, let~$i^\eps\colon H_1(S)\to H_1(S^3\setminus S)$ be defined as follows. Any homology classes in~$H_1(S)$ can be represented by an oriented cycle~$x$ which behaves as
illustrated in Figure~\ref{fig:clasp} whenever crossing a clasp. Then, define~$i^\eps([x])$ as
the class of the~$1$-cycle obtained by pushing~$x$ in the~$\eps_i$-normal direction off~$S_i$ for~$i=1,\dots,\mu$.
Finally, consider the bilinear form
\[
\alpha^\eps\colon H_1(S)\times H_1(S)\to\Z,\quad(x,y)\mapsto\ell k(i^\eps(x),y)\,,
\]
where~$\ell k$ denotes the linking number.
Fix a basis of~$H_1(S)$ and denote by~$A^\eps$ the matrix of~$\alpha^\eps$. Note that for all~$\eps$,~$A^{-\eps}$ is equal to~$(A^\eps)^T$. Using this fact, one easily checks that for any~$\omega=(\omega_1,\dots,\omega_\mu)$ in the~$\mu$-dimensional 
torus~$\mathbb{T}^\mu$, the matrix
\[
H(\omega)=\sum_\eps\prod_{i=1}^\mu(1-\overline{\omega}_i^{\eps_i})\,A^\eps
\]
is Hermitian.

\begin{definition}
The {\em multivariable signature} of the~$\mu$-colored link~$L$ is the function
\[
\sign(L)\colon\mathbb{T}^\mu\to\Z,\quad\omega\mapsto\sign_\omega(L)\,,
\]
where~$\sign_\omega(L)$ is the signature of the Hermitian matrix~$H(\omega)$.
\end{definition}

Note that in the case~$\mu=1$, one clearly gets back the Levine-Tristram signature.
This multivariable generalization turns out not only to be well-defined (i.e. independent of the choice of the~$C$-complex), but it satisfies
all the properties of the Levine-Tristram signature, generalized from oriented links to colored links. Most notably, it
is invariant under concordance of colored links (see~\cite[Section 7]{CF} for details, and~\cite{DFL} for a recent result on this invariant).

Let us compute a couple of easy examples.

\begin{example}
\label{ex:Hopf}
Consider the positive Hopf link~$\mathcal{H}$, as an oriented 1-colored link. A Seifert surface for~$\mathcal{H}$ is given by a Hopf
band, and the corresponding Seifert matrix is~$A=(-1)$. Therefore, the Levine-Tristram signature of the
positive Hopf link is equal to
\[
\sign_\omega(\mathcal{H})=\mathrm{sgn}(-(1-\omega)-(1-\overline{\omega}))=\mathrm{sgn}(-2\mathit{Re}(1-\omega))=-1
\]
for all~$\omega\in S^1\setminus\{1\}$.
On the other hand, consider the Hopf link as a 2-colored link. Clearly, it admits a contractible~$C$-complex, so its 2-variable signature
vanishes:
\[
\sign_{(\omega_1,\omega_2)}(\mathcal{H})=0\quad\text{for all $(\omega_1,\omega_2)\in\mathbb{T}^2$\,.}
\]
\end{example}

\begin{example}
\label{ex:4}
Consider the 2-colored link~$L$ depicted in the left-hand side of Figure~\ref{fig:ex4}. It admits the natural~$C$-complex illustrated in
the right-hand side of this same figure. An easy computation gives~$A^{++}=A^{--}=(-1)$ and~$A^{+-}=A^{-+}=(0)$, leading to
\[
\sign_{(\omega_1,\omega_2)}(L)=-\mathrm{sgn}\left(\mathit{Re}[(1-\omega_1)(1-\omega_2)]\right).
\]
\end{example}

\begin{figure}[tb]
\labellist\small\hair 2.5pt
\pinlabel {$S_1$} at 562 253
\pinlabel {$S_2$} at 715 253
\endlabellist
\centering
\includegraphics[width=0.5\textwidth]{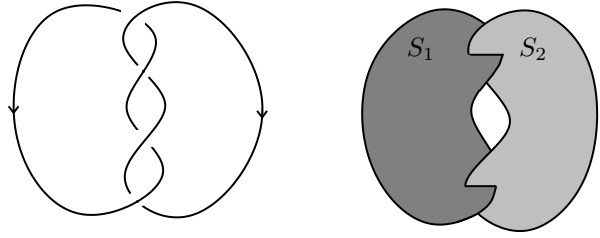}
\caption{The colored link of Example~\ref{ex:4} and a natural~$C$-complex for it.}
\label{fig:ex4}
\end{figure}

From the definition of this invariant, one immediately sees that it is additive under disjoint union.
This translates into the following statement, that we record here for further use.

\begin{lemma}
\label{lemma:sign}
Given any two closable~$\mu$-colored tangles~$\tau_1$ and~$\tau_2$, the equality
\[
\sign_\omega(\widehat{\tau_1\sqcup\tau_2})=\sign_\omega(\widehat{\tau_1})+\sign_\omega(\widehat{\tau_2})
\]
holds for all~$\omega\in\mathbb{T}^\mu$.
\end{lemma}

\medskip

We are now ready to state the goal of this article in more precise terms:

\noindent{\em Given two~$\mu$-colored tangles~$\tau_1,\tau_2\in T_\mu(c,c)$ and an element~$\omega=(\omega_1,\dots,\omega_\mu)\in\mathbb{T}^\mu$, evaluate the additivity defect
\[
\sign_\omega(\widehat{\tau_1\tau_2})-\sign_\omega(\widehat{\tau_1})-\sign_\omega(\widehat{\tau_2}).
\]}

This will be done using an isotropic functor, that will be introduced shortly.

\subsection{The isotropic and Lagrangian categories}
\label{sub:cat}

In this paragraph, we introduce the category~$\mathbf{Isotr}_\Lambda$ of isotropic relations over a ring~$\Lambda$.
This is a slight modification of the category~$\mathbf{Lagr}_\Lambda$ of Lagrangian relations~\cite{CT}, whose definition we also recall
for completeness.

\medbreak

Fix an integral domain~$\Lambda$ endowed with a ring involution~$\lambda \mapsto \overline{\lambda}$.
The main examples to keep in mind are~$\Lambda_\mu:=\Z[t_1^{\pm 1},\dots,t_\mu^{\pm 1}]$ endowed with the involution induced
by~$t_i\mapsto t_i^{-1}$, the field~$\C$ of complex numbers endowed with the complex conjugation, and the field~$\R$ of real
numbers endowed with the trivial involution, which is related to the classical theory of symplectic vector spaces~\cite{Turaev}.

A {\em skew-Hermitian form} on a~$\Lambda$-module~$H$ is a map~$\xi\colon H \times H\rightarrow\Lambda$ such that for
all~$x,y,z \in H$ and all~$\lambda,\lambda'\in\Lambda$,
\begin{enumerate}[(i)]
\item{$\xi(\lambda x + \lambda' y,z)=\lambda \xi(x,z)+\lambda' \xi(y,z),$}
\item{$\xi(x,y)=-\overline{\xi(y,x)}$.}
\end{enumerate}
A {\em Hermitian~$\Lambda$-module}~$H$ is a finitely generated~$\Lambda$-module endowed with a skew-Hermitian form~$\xi$. 
The same module~$H$ with the opposite form~$-\xi$ will be denoted by~$-H$.

Given a submodule~$V$ of a
Hermitian~$\Lambda$-module~$H$, its {\em annihilator} is the submodule
\[
\mathrm{Ann}(V)=\lbrace x \in H \ | \ \xi(v,x)=0 \text{ for all } v \in V \rbrace\,.
\]
A skew-Hermitian form~$\xi$ on~$H$ is {\em non-degenerate} if~$\mathrm{Ann}(H)=0$ and a Hermitian module whose form is non-degenerate will be called a {\em non-degenerate Hermitian module}. A submodule~$V$ of a Hermitian module~$H$ is
{\em isotropic} if~$V \subset \mathrm{Ann}(V)$ or, equivalently, if~$\xi$ vanishes identically on~$V$. A submodule of a Hermitian
module is {\em Lagrangian} if it is equal to its annihilator.

If~$H_1$ and~$H_2$ are Hermitian~$\Lambda$-modules, an {\em isotropic relation}~$N\colon H_1 \Rightarrow H_2$ is an isotropic submodule of~$(-H_1) \oplus H_2$.  For instance, given a Hermitian~$\Lambda$-module~$H$, the {\em diagonal relation}
\[
\Delta_H= \lbrace h \oplus h \in H \oplus H \rbrace
\]
is an isotropic relation~$H \Rightarrow H$. Given two isotropic relations~$N_1\colon H_1 \Rightarrow H_2$ and~$N_2\colon H_2 \Rightarrow H_3$, their {\em composition} is defined as~$N_2 \circ N_1:=N_2N_1\colon H_1 \Rightarrow H_3$ where~$N_2N_1$ denotes the following submodule of~$(-H_1) \oplus H_3$:
\[
N_2N_1 = \lbrace h_1 \oplus h_3 \ | \ h_1 \oplus h_2 \in N_1 \ \text{and} \ h_2 \oplus h_3 \in N_2 \ \text{for a certain} \ h_2 \in H_2 \rbrace\,.
\]
The proof of the following proposition is straightforward.

\begin{proposition}
\label{prop:Isotr}
Hermitian~$\Lambda$-modules, as objects, and isotropic relations as morphisms, form a category~$\mathbf{Isotr}_\Lambda$.
\end{proposition}

Given two Hermitian $\Lambda$-modules~$(H_1,\xi_1)$ and~$(H_2,\xi_2)$, a~$\Lambda$-linear map~$\gamma\colon H_1 \rightarrow H_2$ is said to be {\em unitary\/} if it satisfies~$\xi_2(\gamma(x),\gamma(y))=\xi_1(x,y)$ for all~$x,y \in H_1$. We will denote by~$\widetilde{\mathbf{U}}_\Lambda$ the category of Hermitian~$\Lambda$-modules
and unitary isomorphisms.

Isotropic relations should be understood as a generalization of unitary isomorphisms, in the following sense.
The {\em graph} of a linear map~$\gamma \colon H_1 \rightarrow H_2 $ is the subspace 
\[
\Gamma_{\gamma}= \lbrace v \oplus\gamma(v) \ | \ v \in H_1 \rbrace\subset H_1\oplus H_2\,.
\]
One easily checks that if~$\gamma$ is a unitary isomorphism, then~$\Gamma_\gamma$ is an isotropic
submodule of~$(-H_1)\oplus H_2$, that is, an isotropic relation~$H_1\Rightarrow H_2$. Furthermore, the graph of a composition
of isomorphisms is the composition of the corresponding isotropic relations. In other words, we have the following proposition.

\begin{proposition}
\label{prop:graph}
The map~$\gamma \mapsto \Gamma_\gamma$ defines an embedding of
categories~$\Gamma\colon\widetilde{\mathbf{U}}_\Lambda \hookrightarrow \mathbf{Isotr}_\Lambda.$
\end{proposition}

The corresponding theory for Lagrangian submodules is less straightforward. For this reason, we only recall its main features and refer to \cite[Section 2]{CT} for details and proofs.

Given of a submodule~$A$ of a Hermitian~$\Lambda$-module~$H$, set
\[
\overline{A}= \lbrace x \in H \ | \ \lambda x \in A \ \text{for a non-zero} \ \lambda \in \Lambda \rbrace\,.
\]
If~$H_1$ and~$H_2$ are non-degenerate Hermitian~$\Lambda$-modules, a {\em Lagrangian relation}~$N\colon H_1 \Rightarrow H_2$ is a Lagrangian submodule of~$(-H_1) \oplus H_2$. Given two Lagrangian relations~$N_1\colon H_1 \Rightarrow H_2$ and~$N_2\colon H_2 \Rightarrow H_3$, their {\em composition} is defined as~$N_2 \circ N_1 := \overline{N_2N_1}\colon H_1 \Rightarrow H_3$.
Finally, let~$\mathbf{U}_\Lambda$ be the category of non-degenerate Hermitian~$\Lambda$-modules and unitary isomorphisms.
The following statement is the main result of Section~2 of~\cite{CT}.

\begin{theorem}
\label{thm:Lagr}
Non-degenerate Hermitian~$\Lambda$-modules, as objects, and Lagrangian relations, as morphisms, form a
category~$\mathbf{Lagr}_\Lambda$, and the map~$\gamma \mapsto \Gamma_\gamma$ defines an embedding of categories~$\Gamma\colon\mathbf{U}_\Lambda\hookrightarrow\mathbf{Lagr}_\Lambda$.
\end{theorem}

Let us summarize the content of this paragraph. For any integral domain~$\Lambda$ endowed with a ring involution,
we have the diagram
\begin{displaymath}
\begin{xy}
(0,15)*+{\mathbf{U}_\Lambda}="a";
(20,15)*+{\mathbf{Lagr}_\Lambda}="c";
{\ar@{->}^-{\Gamma} "a";"c"};
(0,0)*+{\widetilde{\mathbf{U}}_\Lambda}="f";
(20,0)*+{\mathbf{Isotr}_\Lambda,}="h";
{\ar@{->}^-{\Gamma} "f";"h"};
{\ar@{->}^{} "a";"f"};
{\ar@{-->}^{} "c";"h"};
\end{xy} 
\end{displaymath}
where the horizontal arrows are the embeddings of categories given by the graph, and the vertical arrows denote the natural
embeddings of categories. The right hand side arrow is dashed because the composition in~$\mathbf{Lagr}_\Lambda$ is not always defined in the same way as in~$\mathbf{Isotr}_\Lambda$. However, if~$\Lambda$ is a field, then~$\overline{A}$ coincides with~$A$ for every subspace~$A\subset H$,  and this arrow does represent a functor.

\subsection{The Maslov index and the Meyer cocycle}
\label{sub:MM}

The Maslov index associates an integer to three isotropic subspaces of a symplectic vector space, while the Meyer cocycle associates an integer to two symplectic automorphisms. The aim of this paragraph is to review these constructions in the spirit of~\cite[Chapter IV.3]{Turaev}, adapting them to the setting of Hermitian complex vector spaces.

Fix a finite dimensional Hermitian complex vector space~$(H,\xi)$, and let~$L_1, L_2$ and~$L_3$ be three
isotropic subspaces of~$H$. Consider the Hermitian form~$f$ defined on~$ (L_1 + L_2) \cap L_3$ as follows:
for~$a,b \in (L_1 + L_2) \cap L_3$, write~$a=a_1+a_2$ with~$a_i\in L_i$, and set~$f(a,b)=\xi(a_2,b)$. One easily checks that~$f$ is a well-defined
Hermitian form.

\begin{definition}
\label{def:Maslov}
The signature of~$f$ is called the {\em Maslov index} of~$L_1,L_2$ and~$L_3$. It will be denoted by~$\mathit{Maslov} (L_1,L_2,L_3)$.
\end{definition}

It should be noted that two other definitions occur in the literature. In~\cite{CTC}, the Maslov index is defined by considering the same Hermitian form but on the quotient~$\frac{(L_1 +L_2) \cap L_3}{(L_1 \cap L_3)+( L_2 \cap L_3)}$. (See also~\cite{Py}.) In~\cite{GG}, the authors consider the space
\[ 
V= \lbrace v_1 \oplus v_2 \oplus v_3 \in L_1 \oplus L_2 \oplus L_3 \ | \ v_1+v_2+v_3=0 \rbrace
\] 
and the Hermitian form defined by sending elements~$a_1 \oplus a_2 \oplus a_3, b_1 \oplus b_2 \oplus b_3 \in V$ to~$\xi(a_2,b_1).$ These definitions are equivalent to ours. This can be seen by noting that if~$f$ is a Hermitian form on~$H$ and~$A$ is a subspace contained in~$\mathrm{Ann}(H)$, then~$f$ descends to a Hermitian form on~$H/A$ whose signature remains unchanged.

We record the following easy lemma for further use.

\begin{lemma}
\label{lemma:Maslov}
The Maslov index satisfies the following properties.
\begin{enumerate}[(i)]
\item{If~$L_1,L_2,L_3$ (resp.~$L'_1,L'_2,L'_3$) are isotropic subspaces of~$H$ (resp.~$H'$) then~$L_1\oplus L'_1,L_2\oplus L'_2,L_3\oplus L'_3$ are isotropic subspaces
of~$H\oplus H'$, and
\[
\mathit{Maslov}(L_1\oplus L'_1,L_2\oplus L'_2,L_3\oplus L'_3)=\mathit{Maslov} (L_1,L_2,L_3)+\mathit{Maslov} (L'_1,L'_2,L'_3)\,.
\]}
\item{For any isotropic subspaces~$L_1,L_2\subset H$,~$\mathit{Maslov}(L_1,L_2,L_2)$ vanishes.}
\item{If~$L_1,L_2,L_3$ are isotropic subspaces of~$H$ and~$\psi$ is a unitary automorphism of~$H$, then~$\psi(L_1),\psi(L_2),\psi(L_3)$ are isotropic subspaces
of~$H$, and
\[
\mathit{Maslov} (\psi(L_1),\psi(L_2),\psi(L_3))=\mathit{Maslov}(L_1,L_2,L_3)\,. 
\]}
\end{enumerate}
\end{lemma}

Let us now introduce the second object of this paragraph.

\begin{definition}
\label{def:Meyer}
The {\em Meyer cocycle} of two unitary automorphisms~$\gamma_1,\gamma_2$ of~$H$ is the integer
\[
\mathit{Meyer}(\gamma_1,\gamma_2) = -\mathit{Maslov} (\Gamma_{\gamma_1^{-1}},\Gamma_{\mathit{id}},\Gamma_{\gamma_2}).
\]
\end{definition}

As for the Maslov index, some equivalent definitions appear in the literature. The Meyer cocyle was originally defined in~\cite{MeyPhd} (see also \cite{Mey}) by considering the space 
\[
M_{\gamma_1,\gamma_2}=\lbrace (v_1,v_2) \ | \  (\gamma_1^{-1}-\mathit{id})v_1=(\mathit{id}-\gamma_2)v_2 \rbrace
\] 
and taking the signature of the bilinear form~$B$ on~$M_{\gamma_1,\gamma_2}$ obtained by setting 
\[
B(v,w)=\xi(v_1+v_2,\gamma_1^{-1}(w_1)-w_1)
\]
for~$v=(v_1,v_2)$ and~$w=(w_1,w_2) \in M_{\gamma_1,\gamma_2}$. For computational purposes, the most practical definition of the Meyer cocycle is given in~\cite{GG}: the authors consider the space
\[
E_{\gamma_1,\gamma_2}=\mathit{Image}(\gamma_1^{-1}-\mathit{id})\cap \mathit{Image}(\mathit{id}-\gamma_2)
\]
and take the signature of the Hermitian form obtained by setting~$b(e,e')=\xi(x_1+x_2,e')$
for~$e=\gamma_1^{-1}(x_1)-x_1=x_2-\gamma_2(x_2)\in E_{\gamma_1,\gamma_2}$. It can be checked that these definitions are equivalent to the one we gave in terms of the Maslov index.

Let us show how to use the latter definition on a couple of explicit examples.

\begin{example}
\label{ex:MM1}
Let~$\omega$ be any complex number of modulus 1, and consider the one-dimensional complex vector space~$H=\C$ endowed with
the skew-Hermitian form given by the matrix
\[
\xi(\omega)=\left(\omega-\overline{\omega}\right)\,.
\]
The automorphism~$\gamma$ of~$H$ given by multiplication by~$-\omega$ is unitary with respect to the matrix~$\xi(\omega)$.
Since~$(\gamma^{-1}-\id)(-\omega)=(\id-\gamma)(1)=1+\omega=:e$, we get~$E_{\gamma\gamma}=\C e$ and
\[
b(e,e)=(1-\omega)(\omega-\overline{\omega})(1+\overline{\omega})=\|\omega-\overline{\omega}\|^2.
\]
This leads to
\begin{equation}
\label{equ:MM1}
\mathit{Meyer}(\gamma,\gamma)=\begin{cases}1 &\mbox{if } \omega\neq \pm 1; \\ 0 & \mbox{if } \omega=\pm 1.\end{cases}
\end{equation}
\end{example}

\begin{example}
\label{ex:MM2}
Let~$\omega$ be a complex number of modulus 1, and consider the two-dimensional complex vector space~$H$ endowed with
the skew-Hermitian form given by the matrix
\[
\xi(\omega)=\begin{pmatrix}
\omega-\overline{\omega} & 1-\omega\cr -1+\overline{\omega}&\omega-\overline{\omega}\cr
\end{pmatrix}\,.
\]
The automorphism~$\gamma_1$ of~$H$ given by the matrix
\[
\gamma_1=\begin{pmatrix}-\omega & 1\cr 0&1\cr\end{pmatrix}
\]
is unitary with respect to~$\xi(\omega)$, so~$\gamma_2:=\gamma_1^2$ is unitary as well. If~$\omega=1$,~$\xi(\omega)$ vanishes. For~$\omega\neq 1$, an immediate computation yields~$E_{\gamma_1,\gamma_2}=\C e$, with
\[
e:=\begin{pmatrix}1\cr 0\end{pmatrix}=(\gamma_1^{-1}-\mathit{id})\begin{pmatrix}0\cr \omega\end{pmatrix}=(\mathit{id}-\gamma_2)\begin{pmatrix}0\cr (\omega-1)^{-1}\end{pmatrix}
\]
This leads to
\[
b(e,e)=\begin{pmatrix}0&\omega+(\omega-1)^{-1}\end{pmatrix}\begin{pmatrix}\omega-\overline{\omega} & 1-\omega\cr -1+\overline{\omega}&\omega-\overline{\omega}\end{pmatrix}\begin{pmatrix}1\cr 0\end{pmatrix}=1-2\mathit{Re}(\omega)\,,
\]
so
\begin{equation}
\label{equ:MM2}
\mathit{Meyer}(\gamma_1,\gamma_2)=\mathrm{sgn}(1-2\mathit{Re}(\omega)).
\end{equation}
\end{example}

\begin{example}
\label{ex:MM3}
Fix~$(\omega_1,\omega_2)\in\mathbb{T}^2$, and consider the one-dimensional complex vector space endowed with
the skew-Hermitian form given by the matrix
\[
\xi(\omega_1,\omega_2)
=\left((\omega_1-\overline{\omega}_1)+(\omega_2-\overline{\omega_2})-(\omega_1\omega_2-\overline{\omega}_1\overline{\omega}_2)\right)\,.
\]
The automorphism~$\gamma$ given by multiplication by~$\omega_1\omega_2$ is unitary with respect to~$\xi(\omega_1,\omega_2)$. Since~$(\gamma^{-1}-\id)(\omega_1\omega_2)=(\id-\gamma)(1)=1-\omega_1\omega_2=:e$, we
get~$E_{\gamma,\gamma}=\C e$, and
\begin{align*}
\mathit{Meyer}(\gamma,\gamma) &=\mathrm{sgn}(b(e,e))=\mathrm{sgn}((\omega_1\omega_2+1)\xi(\omega_1,\omega_2)(1-\overline{\omega}_1\overline{\omega}_2))\\
&=\mathrm{sgn}\left(\mathit{Re}\left[(1-\omega_1)(1-\omega_2)(1-\omega_1\omega_2)\right]\right)\,.
\end{align*}
\end{example}

Let us conclude this paragraph with one last observation. Note that the definition of the Meyer cocycle depends on whether one chooses~$\xi \oplus -\xi$ or~$-\xi \oplus \xi$ as the skew-Hermitian form on~$H \oplus H$. We chose the latter, which explains the presence of the minus sign in our definition, a sign which does not appear in~\cite{GG}.

\subsection{The isotropic functor}
\label{sub:functor}

The paper~\cite{CT} deals with an extension of the Burau (and more generally, the colored Gassner) representation from braids
to (colored) tangles using free abelian coverings of the tangle exterior. We recall this construction in subsection~\ref{sub:Lagr}
below, but let us briefly describe its main features.

This extension takes the form of a functor~$\mathcal{F}\colon\mathbf{Tangles}_\mu\to\mathbf{Lagr}_{\Lambda_\mu}$, where~$\Lambda_\mu$ denotes the ring~$\Z[t_1^{\pm 1},\dots,t_\mu^{\pm 1}]$ endowed with the involution induced
by~$t_i\mapsto t_i^{-1}$. This functor fits in the commutative diagram
\[
\begin{xy}
(0,15)*+{\mathbf{Braids}_\mu}="a";
(25,15)*+{\mathbf{Tangles}_\mu}="c";
{\ar@{->}^-{} "a";"c"};
(0,0)*+{\mathbf{U}_{\Lambda_\mu}}="f";
(25,0)*+{\mathbf{Lagr}_{\Lambda_\mu},}="h";
{\ar@{->}^-{\Gamma} "f";"h"};
{\ar@{->}^{} "a";"f"};
{\ar@{->}^{\mathcal{F}} "c";"h"};
\end{xy} 
\]
where the horizontal arrows are the embeddings of categories described in Sections~\ref{sub:tangle} and~\ref{sub:cat}. It is an extension
of the Burau and Gassner representations in the following sense. If~$\alpha$ is an~$n$-strand~$\mu$-colored braid,
then~$\mathcal{F}(\alpha)$ is the graph of a unitary automorphism of a~$\Lambda_\mu$-module of rank~$n-1$,
which is free if~$\mu\le 2$, and whose localization with respect to the multiplicative set~$S\subset\Lambda_\mu$
generated by~$t_1-1,\dots,t_\mu-1$ is always free over the localized ring~$\Lambda_S:=S^{-1}\Lambda_\mu=\Z[t_1^{\pm 1},\dots,t_\mu^{\pm 1},(t_1-1)^{-1},\dots,(t_\mu-1)^{-1}]$. Then, with respect to the correct basis, the matrix of this automorphism in the case~$\mu=1$ (resp.~$\mu=n$) is nothing but the matrix of the reduced Burau
(resp. Gassner) representation of the braid group~$B_n$ (resp. pure braid group~$P_n$) at~$\alpha$.

More generally, for any sequence~$c\colon\{1,\dots,n\}\to\{\pm 1,\dots,\pm\mu\}$, the restriction of~$\mathcal{F}$ defines a
unitary representation of the corresponding braid group
\[
\mathcal{B}_t\colon B_c\longrightarrow U_{n-1}(\Lambda_S)\,,
\]
which we call the {\em colored Gassner representation\/}. (The terminology ``colored Burau representation'' is already used
for different (but related) objects~\cite{Morton,L-L}.)

Let us recall the explicit formulae for this representation in the case~$\mu=1$, referring to~\cite[p. 763]{CT} for the easy proof.

\begin{example}
\label{ex:burau}
Let~$\eps=(\eps_1,\dots,\eps_n)$ be any sequence of signs whose sum does not vanish, and let~$\sigma_i$ denote the standard
generator of the corresponding braid group~$B_\eps$, for~$i=1,\dots,n-1$. Then, with respect to a preferred basis, matrices for the Burau representation of~$\sigma_i$ are given by
\begin{equation}
\label{equ:burau}
\mathcal{B}_t(\sigma_1)=\begin{pmatrix}-t^{\eps_2}&1\cr 0&1\end{pmatrix}\oplus I_{n-3},\quad 
\mathcal{B}_t(\sigma_{n-1})=I_{n-3}\oplus\begin{pmatrix}1&0\cr t^{\eps_n}&-t^{\eps_n}\end{pmatrix}\,,
\end{equation}
\[
\mathcal{B}_t(\sigma_i)=I_{i-2}\oplus\begin{pmatrix}1&0&0\cr t^{\eps_{i+1}}&-t^{\eps_{i+1}}&1\cr 0&0&1\end{pmatrix}\oplus I_{n-i-2}\;\;\hbox{ for }\;2\le i\le n-2.
\]
\end{example}

The general formula for the colored Gassner representation can be quite complicated, so let us only give one easy 2-variable example,
which will be computed in subsection~\ref{sub:Lagr}.

\begin{example}
\label{ex:gassner}
The Gassner representation of the pure braid group~$P_2$ maps the standard generator~$A_{12}$ to the
matrix~$\mathcal{B}_{(t_1,t_2)}(A_{12})=(t_1t_2)$.
\end{example}

\begin{remark}
Note that the restriction of this functor to string links (i.e. tangles all of whose components
join~$D^2\times\{0\}$ and~$D^2\times\{1\}$) gives a colored generalization of the extension of the Gassner representation first introduced by Le Dimet~\cite{LeD}, and thoroughly studied by Kirk {\em et al.\/}~\cite{KLW}. However, we shall not investigate this special case in the present article.
\end{remark}

For the computations, it is also crucial to know the form of the skew-Hermitian~$\Lambda_\mu$-module~$\mathcal{F}(c)$
for any given~$c\colon\{1,\dots,n\}\to\{\pm 1,\dots,\pm\mu\}$. The
explicit general formula for the associated skew-Hermitian form~$\xi_c$ can be quite involved, so let us focus on a couple of examples. These computations will be explained in subsection~\ref{sub:Lagr}.

\begin{example}
\label{ex:form1}
In the case~$\mu=1$,~$c$ is a sequence~$(\eps_1,\dots,\eps_n)$ of~$\pm 1$'s. With respect to the preferred basis of~$\mathcal{F}(c)$ considered in
Example~\ref{ex:burau}, the~$\Z[t,t^{-1}]$-valued form~$\xi_c$ is given by the matrix
\begin{equation}
\label{equ:form1}
\begin{pmatrix}
  \frac{1}{2}(\varepsilon_1+\varepsilon_2)(t-t^{-1})& 1-t^{\varepsilon_2} & 0 & \hdots & 0\\
  t^{-\varepsilon_2}-1 & \frac{1}{2}(\varepsilon_2+\varepsilon_3)(t-t^{-1}) &  & \ddots & \vdots \\
  0 &  & \ddots  &  & 0 \\
  \vdots & \ddots &  & & 1-t^{\varepsilon_n}  \\
  0 & \hdots & 0 &t^{-\varepsilon_n}-1 & \frac{1}{2}(\varepsilon_{n-1}+\varepsilon_n)(t-t^{-1})
 \end{pmatrix}.
\end{equation}
Note that in the classical case~$\eps=(1,\dots,1)$, setting~$s=t^2$, we obtain a multiple of the matrix 
\[
\begin{pmatrix}
  s+s^{-1}& -s & 0 & \hdots & 0\\
  -s^{-1} & s+s^{-1} &  & \ddots & \vdots \\
  0 &  & \ddots  &  & 0 \\
  \vdots & \ddots &  & s+s^{-1} & -s  \\
  0 & \hdots & 0 & -s^{-1} & s+s^{-1}
 \end{pmatrix}\,,
\]
which is the form with respect to which the Burau representation was shown to be unitary by Squier~\cite{Squier}.
\end{example}

\begin{example}
\label{ex:form22}
Let us consider the case~$n=\mu=2$ and~$c=(1,2)$. One can show that with respect to some natural basis, the corresponding~$\Lambda_2$-valued
skew-Hermitian form~$\xi_c$ is given by the matrix
\[
\left((t_1-t_1^{-1})+(t_2-t_2^{-1})-(t_1t_2-t_1^{-1}t_2^{-1})\right)\,.
\]
\end{example}

\medskip

For the purpose of the present paper, we need to extend the {\em evaluation\/} of the colored Gassner representation
at~$t=\omega\in\mathbb{T}^\mu$ from braids to tangles. This requires the definition, for each~$\omega$, of a modified functor~$\mathcal{F}_\omega$,
which is constructed in the same manner as~$\mathcal{F}$, but using finite abelian branched covers instead of free abelian ones.
In this section, we only state its main properties, postponing to subsection~\ref{sub:isotr} its construction and the proof of these statements.

For each torsion element~$\omega$ in~$\mathbb{T}^\mu$, we construct a functor~$\mathcal{F}_\omega\colon\mathbf{Tangles}_\mu\to\mathbf{Isotr}_\C$ which fits in the commutative diagram
\[
\begin{xy}
(0,15)*+{\mathbf{Braids}_\mu}="a";
(30,15)*+{\mathbf{Tangles}_\mu}="c";
{\ar@{->}^-{} "a";"c"};
(0,0)*+{\mathbf{\widetilde{U}}_\C}="g";
(30,0)*+{\mathbf{Isotr}_\C\,.}="h";
{\ar@{->}^-{\Gamma} "g";"h"};
{\ar@{->}^{} "a";"g"};
{\ar@{->}^{\mathcal{F}_\omega} "c";"h"};
\end{xy} 
\]
(See Theorem~\ref{thm:isotropicfunctor} below.)
In general, the object~$\mathcal{F}_\omega(c)$ of~$\mathbf{Isotr}_\C$
associated to a sequence~$c$ of~$\pm 1,\dots,\pm\mu$ is a complex vector space endowed with a skew-Hermitian form which
can be degenerate. However, let us assume that each coordinate~$\omega_i$ of~$\omega$ is of order~$k_i>1$ with these~$k_i$'s
pairwise coprime. If the sequence~$c$ is such that for all~$i=1,\dots,\mu$,~$\ell(c)_i:=\sum_{j;c_j=\pm i}\mathrm{sgn}(c_j)$ does not vanish and is coprime to~$k_i$,
then~$\mathcal{F}_\omega(c)$ is non-degenerate. To be more precise, let us denote by~$\mathbf{Tangles}^\omega_\mu$
(resp.~$\mathbf{Braids}^\omega_\mu$) the full subcategory of~$\mathbf{Tangles}_\mu$ (resp.~$\mathbf{Braids}_\mu$) given by sequences fulfilling the condition above.
Then, the restriction of~$\mathcal{F}_\omega$ to~$\mathbf{Tangles}^\omega_\mu$ defines a functor which fits in the commutative diagram
\[
\begin{xy}
(0,15)*+{\mathbf{Braids}^\omega_\mu}="a";
(30,15)*+{\mathbf{Tangles}^\omega_\mu}="c";
{\ar@{->}^-{} "a";"c"};
(0,0)*+{\mathbf{U}_\C}="g";
(30,0)*+{\mathbf{Lagr}_\C.}="h";
{\ar@{->}^-{\Gamma} "g";"h"};
{\ar@{->}^{} "a";"g"};
{\ar@{->}^{\mathcal{F}_\omega} "c";"h"};
\end{xy} 
\]
(See Proposition~\ref{prop:FLagr}.)

Finally, one might wonder how the functors~$\mathcal{F}$ and~$\mathcal{F}_\omega$ are related. As far as objects are concerned, the
answer is very simple: for any~$c\colon\{1,\dots,n\}\to\{\pm 1,\dots,\pm\mu\}$, a matrix for the skew-Hermitian form of the complex vector
space~$\mathcal{F}_\omega(c)$ can be obtained by evaluating at~$t=\omega$ a matrix for the skew-Hermitian form~$\xi_c$ of the
(localized) module~$\mathcal{F}(c)$ (see Proposition~\ref{prop:form}). In particular, Examples~\ref{ex:form1} and~\ref{ex:form22} lead to the following results.

\begin{example}
\label{ex:form3}
In the case~$\mu=1$, a matrix for the skew-Hermitian complex form given by~$\mathcal{F}_\omega(c)$ is
\begin{equation}
\label{equ:form3}
\begin{pmatrix}
\frac{1}{2}(\varepsilon_1+\varepsilon_2)(\omega-\overline{\omega})& 1-\omega^{\varepsilon_2} & 0 & \hdots & 0\\
\overline{\omega}^{\varepsilon_2}-1 & \frac{1}{2}(\varepsilon_2+\varepsilon_3)(\omega-\overline{\omega}) &  & \ddots & \vdots \\
0 &  & \ddots  &  & 0 \\
\vdots & \ddots &  & & 1-\omega^{\varepsilon_n}  \\
0 & \hdots & 0 &\overline{\omega}^{\varepsilon_n}-1 & \frac{1}{2}(\varepsilon_{n-1}+\varepsilon_n)(\omega-\overline{\omega})
\end{pmatrix}.
\end{equation}
\end{example}

\begin{example}
\label{ex:form4}
In the case~$n=\mu=2$ and~$c=(1,2)$, a matrix for the skew-Hermitian complex form associated to~$\mathcal{F}_\omega(c)$ is
\[
\xi_c(\omega_1,\omega_2)=\left((\omega_1-\overline{\omega}_1)+(\omega_2-\overline{\omega}_2)-(\omega_1\omega_2-\overline{\omega}_1\overline{\omega}_2)\right)\,.
\]
\end{example}

For morphisms, the answer is trickier. Roughly speaking,~$\mathcal{F}_\omega$ is the evaluation of~$\mathcal{F}$ at~$t=\omega$ ``whenever that makes sense''. The precise statement is slightly technical, but the following special case will suffice for the purpose of the present discussion
(see subsection~\ref{sub:rel} for details).
Let us say that a tangle is {\em topologically trivial\/} if its exterior is homeomorphic to the exterior of a trivial braid. It turns out that
if a~$(c,c')$-tangle~$\tau$ is topologically trivial, then working over the localized ring~$\Lambda_S$, one obtains that~$\mathcal{F}(\tau)$
is a free submodule of the free~$\Lambda_S$-module~$(-\mathcal{F}(c))\oplus\mathcal{F}(c')$, so this inclusion can be encoded by a matrix~$M(t)$
with coefficients in~$\Lambda_S$.
In this case,~$\mathcal{F}_\omega(\tau)$ is equal to the complex subspace of~$(-\mathcal{F}_\omega(c))\oplus\mathcal{F}_\omega(c')$
encoded by the matrix~$M(\omega)$. In particular, since a braid~$\alpha$ is topologically trivial,~$\mathcal{F}_\omega(\alpha)$ is
nothing but the graph of the unitary automorphism given by the evaluation~$\mathcal{B}_\omega(\alpha)$ of a matrix~$\mathcal{B}_t(\alpha)$ of the colored Gassner representation at~$t=\omega$.

\subsection{Statement of the result and examples}
\label{sub:thm}

We are finally ready to state our main result in a precise way, and to illustrate it with examples.

Let~$\mathbb{T}^\mu_P$ denote the dense subset of the~$\mu$-dimensional torus~$\mathbb{T}^\mu$ composed of the elements~$\omega=(\omega_1,\dots,\omega_\mu)$ such that
the orders~$k_1,\dots,k_\mu$ of~$\omega_1,\dots,\omega_\mu$ are greater than~$1$ and pairwise coprime. Recall that for a coloring~$c\colon\lbrace 1,\dots,n\rbrace\to\lbrace\pm 1,\dots,\pm\mu\rbrace$,
we defined~$\ell(c)$ as the element in~$\Z^\mu$ whose~$i^{\mathit{th}}$ coordinate is given by~$\ell(c)_i=\sum_{j;c_j=\pm i}\mathrm{sgn}(c_j)$. 
Given a coloring~$c$ with~$\ell(c)_i\neq 0$ for all~$i$, let~$\mathbb{T}^\mu_c$ denote the dense subset of~$\mathbb{T}^\mu$ given by the
elements~$\omega$ such that for all~$i$,~$k_i$ is coprime to~$\ell(c)_i$.

\begin{theorem}
\label{thm:main}
For any~$c$ such that~$\ell(c)$ is nowhere zero and for any~$(c,c)$-colored tangles~$\tau_1$ and~$\tau_2$,
the equality
\[
\mathit{sign}_\omega(\widehat{\tau_1\tau_2})-\mathit{sign}_\omega(\widehat{\tau_1})-\sign_\omega(\widehat{\tau_2})=\mathit{Maslov}(\mathcal{F}_\omega( \overline{\tau}_1), \Delta ,\mathcal{F}_\omega(\tau_2))
\]
holds for all~$\omega$ in the dense subset~$\mathbb{T}^\mu_c\cap\mathbb{T}^\mu_P$ of the torus~$\mathbb{T}^\mu$.
\end{theorem}

\begin{remark}
\label{rem:trick}
Note that the conditions of~$\ell(c)$ being nowhere zero and~$\omega$ belonging to~$\mathbb{T}^\mu_c$ are not restrictive.
Indeed, assume that we want to evaluate the integer
\[
\delta_\omega(\tau_1,\tau_2):=\mathit{sign}_\omega(\widehat{\tau_1\tau_2})-\mathit{sign}_\omega(\widehat{\tau_1})-\sign_\omega(\widehat{\tau_2})
\]
for a given~$\omega\in\mathbb{T}_P^\mu$ and~$(c,c)$-tangles~$\tau_1,\tau_2$. Then, one can always find a sequence~$c'$
such that~$\ell(c\sqcup c')$ is nowhere zero and~$\omega$ belongs to~$\mathbb{T}^\mu_{c\sqcup c'}$. By Lemma~\ref{lemma:sign} together with
the fact that the signature of any colored unlink vanishes, we have the equality
\[
\delta_\omega(\tau_1,\tau_2)=\delta_\omega(\tau_1\sqcup\mathit{id}_{c'},\tau_2\sqcup\mathit{id}_{c'})
\]
for any~$c'$ and any~$\omega$. Therefore, Theorem~\ref{thm:main} allows us to compute this defect for any pair of colored tangles
and any~$\omega\in\mathbb{T}_P^\mu$.
\end{remark}

The special case~$\mu=1$, which corresponds to oriented tangles and the Levine-Tristram signature, takes the following form.

\begin{corollary}
\label{cor:oriented}
For any sequence~$\eps$ of~$\pm 1$'s whose sum~$\ell(\eps)$ does not vanish, and for any~$(\eps,\eps)$-tangles~$\tau_1$ and~$\tau_2$, the equality
\[
\mathit{sign}_\omega(\widehat{\tau_1\tau_2})-\mathit{sign}_\omega(\widehat{\tau_1})-\sign_\omega(\widehat{\tau_2})=\mathit{Maslov}(\mathcal{F}_\omega( \overline{\tau}_1), \Delta ,\mathcal{F}_\omega(\tau_2))
\]
holds for all~$\omega\in S^1\setminus\{1\}$ whose order is coprime to~$\ell(\eps)$.
\end{corollary}

Note that by Remark~\ref{rem:trick}, this actually allows us to compute the additivity defect of the Levine-Tristram signature
evaluated at any root of unity.

Let us now come back to the general multivariable case, but specialized to colored braids.
As explained in subsection~\ref{sub:functor}, if~$\alpha$ is a colored braid, then~$\mathcal{F}_\omega(\alpha)=\Gamma_{\mathcal{B}_\omega(\alpha)}$, the graph of the colored Gassner representation
evaluated at~$\omega$. Moreover, the horizontal reflection~$\overline{\alpha}$ of~$\alpha$ is nothing but its inverse.
Finally, as stated in subsection~\ref{sub:MM}, the Meyer cocycle and Maslov index are related by the equality
\[
\mathit{Maslov} (\Gamma_{\gamma_1^{-1}},\Delta,\Gamma_{\gamma_2})=
\mathit{Maslov} (\Gamma_{\gamma_1^{-1}},\Gamma_{\mathit{id}},\Gamma_{\gamma_2})=-\mathit{Meyer}(\gamma_1,\gamma_2).
\]
Consequently, we obtain the following corollary. 

\begin{corollary}
\label{cor:braid}
For any~$c$ such that~$\ell(c)$ is nowhere zero and for any two colored braids~$\alpha,\beta\in B_c$, the equality
\[
\mathit{sign}_\omega(\widehat{\alpha\beta})-\mathit{sign}_\omega(\widehat{\alpha})-\sign_\omega(\widehat{\beta})=
-\mathit{Meyer}(\mathcal{B}_\omega(\alpha), \mathcal{B}_\omega(\beta))
\]
holds for all~$\omega$ in~$\mathbb{T}^\mu_c\cap\mathbb{T}^\mu_P$.
\end{corollary}

Obviously, Remark~\ref{rem:trick} applies to this particular case, so we can compute the multivariable
signature evaluated at any~$\omega\in\mathbb{T}^\mu_P$. But more can be said in this case.

\begin{remark}
\label{rem:dense}
Observe that both sides of the equality in Corollary~\ref{cor:braid} are defined for all~$\omega\in\mathbb{T}^\mu$. Moreover,
it is known that the multivariable signature~$\sign(L)$ is constant on the connected components of~$\mathbb{T}^\mu\setminus V_L$,
where~$V_L$ denotes the intersection of the torus with the algebraic variety defined by the (colored) Alexander ideal of~$L$. (See~\cite[Section 4]{CF} for the precise statement.) Therefore, the left-hand side of this equality is constant on the connected
components of~$\mathbb{T}^\mu\setminus V$, where~$V$ is some algebraic variety defined by the Alexander ideals of the closures
of~$\alpha\beta$,~$\alpha$, and~$\beta$. A similar (but so far, less precise) statement can be proved for the right-hand side of this
equality: it is constant on the connected components of~$\mathbb{T}^\mu\setminus V'$, where~$V'$ is some algebraic variety.
Since Corollary~\ref{cor:braid} establishes the equality of these functions on a dense subset of the torus, we can conclude that they
coincide on the {\em open\/} dense subset~$\mathbb{T}^\mu\setminus (V\cup V')$.
\end{remark}

\begin{remark}
\label{rem:alg}
Let~$\alpha_\pm\in B_c$ be an arbitrary colored braid, and let~$\alpha_\mp\in B_c$ be obtained from~$\alpha_\pm$ by a crossing change.
Up to conjugation in the group~$B_c$, we have~$\alpha_\mp=\sigma_i^{\mp 2}\alpha_\pm$ for some~$i$. Therefore, Corollary~\ref{cor:braid} gives
\[
\mathit{sign}_\omega(\widehat{\alpha_\mp})-\mathit{sign}_\omega(\widehat{\alpha_\pm})=\sign_\omega(\widehat{\sigma_i^{\mp 2}})-\mathit{Meyer}(\mathcal{B}_\omega(\alpha_\pm), \mathcal{B}_\omega(\sigma_i^{\mp 2}))\,.
\]
Note that~$\sign_\omega(\widehat{\sigma_i^{\mp 2}})=\pm 1$ if both strands involved in the crossing have the same color, and it vanishes otherwise
(recall Example~\ref{ex:Hopf}).
Furthermore, the explicit (sparse) form of~$\mathcal{B}_\omega(\sigma_i^{\mp 2})$ implies that this Meyer cocycle is in~$\{-1,0,1\}$. This gives an explicit formula relating
the multivariable signature of two links related by a crossing change.

As a first consequence, we obtain a new proof of the well-known fact that (half of) the signature provides a lower bound for the
unlinking number of a link. (See e.g.~\cite[Proposition~5.3]{GG} for the univariate case, and~\cite[Section~5]{CF} for the general case.) Furthermore, since any colored link can be realized as the closure of a colored braid, and any colored braid can be transformed by
crossing changes into a braid whose closure is a trivial link, this provides a new algorithm for the computation of the multivariable signature of any colored link.
\end{remark}

\begin{remark}
Recall that given a coloring~$c\colon\lbrace 1,\dots,n\rbrace\to\lbrace\pm 1,\dots,\pm\mu\rbrace$, the corresponding colored Gassner representation evaluated at~$t=\omega$
is a homomorphism~$\mathcal{B}_\omega$ from the associated colored braid group~$B_c$ to the group of unitary automorphisms of a Hermitian complex vector space of dimension~$n-1$.
Since the Meyer cocycle evaluated on this group is the signature of a Hermitian form on a space of dimension at most~$2(n-1)$, Corollary~\ref{cor:braid} implies the inequality
\[
\big|\mathit{sign}_\omega(\widehat{\alpha\beta})-\mathit{sign}_\omega(\widehat{\alpha})-\sign_\omega(\widehat{\beta})\big|\le2(n-1)
\]
for all~$\alpha,\beta\in B_c$ and~$\omega$ in~$\mathbb{T}^\mu_c\cap\mathbb{T}^\mu_P$. In particular, for any such~$\omega$,
the map~$B_c\to\R$, $\alpha\mapsto\mathit{sign}_\omega(\widehat{\alpha})$ is a {\em quasimorphism\/}.
\end{remark}

Note that even the intersection of Corollaries~\ref{cor:oriented} and~\ref{cor:braid}, i.e. the case of oriented braids and the Levine-Tristram signature, is slightly more general than the main theorem of Gambaudo and Ghys stated in the introduction, as we allow the strands of the braids to be oriented in different directions.

With all these positive results, one might wonder whether the equation of Corollary~\ref{cor:braid} does not hold true for
all~$\omega\in\mathbb{T}^\mu$. This is {\em not\/} the case, even for the classical signature,
as demonstrated by the following simple example.

\begin{example}
\label{ex:comp1}
Consider the classical case~$\mu=1$ and~$c=(1,1)$, and let~$\alpha=\beta$ be the standard (positive) generator of the braid group~$B_2$.
Since~$\widehat{\alpha\beta}$ is the positive Hopf link~$\mathcal{H}$ and~$\widehat{\alpha}$ the unknot, Example~\ref{ex:Hopf} leads
to
\[
\mathit{sign}_\omega(\widehat{\alpha\beta})-\mathit{sign}_\omega(\widehat{\alpha})-\sign_\omega(\widehat{\beta})
=\sign_\omega(\mathcal{H})=-1
\]
for all~$\omega\in S^1\setminus\{1\}$.
On the other hand, we know by~$(\ref{equ:burau})$ that a matrix for the reduced Burau representation evaluated at~$\alpha$
is equal to~$\mathcal{B}_t(\alpha)=(-t)$, which is unitary with respect to the form~$(t-t^{-1})$ (recall~$(\ref{equ:form1})$).
By~$(\ref{equ:MM1})$, we have
\[
\mathit{Meyer}(\mathcal{B}_\omega(\alpha),\mathcal{B}_\omega(\beta))=\begin{cases}1 &\mbox{if } \omega\neq \pm 1; \\ 0 & \mbox{if } \omega=\pm 1.\end{cases}
\]
Therefore, we see that the equality in Corollary~\ref{cor:braid} is satisfied for all~$\omega\in S^1\setminus\{-1\}$, but {\em not\/}
for~$\omega=-1$, whose order is not coprime to~$n=2$. This shows that, even in the most basic case of the classical (Murasugi)
signature, one does need~$n$ to be odd for this equality to hold.
\end{example}

\begin{example}
\label{ex:comp2}
Let us now consider the case~$\mu=1$,~$n=2$, and~$\eps=(+1,-1)$. By~$(\ref{equ:form3})$, the associated skew-Hermitian
form~$\xi_c(\omega)$ vanishes for all~$\omega$. Hence, all the Meyer cocycles computed in this setting will vanish as well.
On the other hand, one can of course form oriented links with non-vanishing signature by closing up braids in~$B_c$. (The Hopf
link is the simplest example.) This shows that the assumption~$\ell(c)\neq 0$ is necessary for our result to hold, even in the simple
case of oriented braids.
\end{example}

We conclude this section with some more examples.

\begin{example}
\label{ex:comp3}
Let us compute the Levine-Tristram signature of the positive trefoil knot~${T}$ without using any Seifert surface.
Since we want to make sure we get the correct value at~$\omega=-1$, consider the standard generator~$\sigma_1$ in~$B_3$
(and not~$B_2$). Applying Corollary~\ref{cor:braid} to~$\alpha=\sigma_1$ and~$\beta=\sigma_1^2$, and using the
fact that~$\widehat{\alpha\beta}={T}$,~$\widehat{\beta}=\mathcal{H}$ whose signature is~$-1$,
and~$\widehat{\alpha}$ is the unknot whose signature vanishes, we get
\[
\mathit{sign}_\omega({T})+1=-\mathit{Meyer}(\gamma,\gamma^2)\,,\quad\text{where}\quad
\gamma=\mathcal{B}_\omega(\sigma_1)=\begin{pmatrix}-\omega & 1\cr 0&1\cr\end{pmatrix}
\]
by~$(\ref{equ:burau})$. Using~$(\ref{equ:form3})$, the relevant form is given by the matrix
\[
\begin{pmatrix}\omega-\overline{\omega} & 1-\omega\cr -1+\overline{\omega}&\omega-\overline{\omega}\end{pmatrix}\,,
\]
so~$(\ref{equ:MM2})$ gives
\[
\mathit{sign}_\omega({T})=-1+\mathrm{sgn}(2\mathit{Re}(\omega)-1).
\]
This turns out to be the correct value at all~$\omega\in S^1$, even at the roots of unity of order divisible by~$3$.
\end{example}

\begin{example}
\label{ex:comp4}
Consider the 2-colored link~$L$ illustrated in Figure~\ref{fig:ex4}. Clearly, it is the closure of the square of the standard generator~$A_{12}$ of the pure braid group~$P_2$. Applying Corollary~\ref{cor:braid} to~$\alpha=\beta=A_{12}\in P_2$,
and using the fact that~$\widehat{A_{12}}$ is the 2-colored Hopf link whose 2-variable signature vanishes, we get
\[
\mathit{sign}_{(\omega_1,\omega_2)}(L)=-\mathit{Meyer}(\gamma,\gamma)\,,\quad\text{where}\quad
\gamma=\mathcal{B}_{(\omega_1,\omega_2)}(A_{12})=(\omega_1\omega_2)
\]
by Example~\ref{ex:gassner}. Using Examples~\ref{ex:form4} and~\ref{ex:MM3}, we obtain
\[
\mathit{sign}_{(\omega_1,\omega_2)}(L)=-\mathrm{sgn}\left(\mathit{Re}\left[(1-\omega_1)(1-\omega_2)(1-\omega_1\omega_2)\right]\right).\
\]
Comparing this with Example~\ref{ex:4}, it remains to see when
\[
\mathit{Re}\left[(1-\omega_1)(1-\omega_2)\right]\quad\text{and}\quad\mathit{Re}\left[(1-\omega_1)(1-\omega_2)(1-\omega_1\omega_2)\right]
\]
have the same sign. Writing~$\omega_1=e^{i \theta_1}$ and~$\omega_2=e^{i\theta_2}$, one easily checks that this is the case if and only if~$\omega_1\omega_2\neq 1$.
In particular, this holds when the orders of the roots of unity~$\omega_1,\omega_2$ are coprime, as predicted by Corollary~\ref{cor:braid}.
However, this example shows that the hypothesis~$\omega\in\mathbb{T}^\mu_P$ is necessary for this result to hold.
\end{example}


\section{Algebraic and topological preliminaries}
\label{sec:prelim}

The aim of this section is to introduce the tools needed to prove our main result. In subsection~\ref{sub:eval},  we deal with generalized eigenspaces while in subsection~\ref{sub:sign2}, we review signatures of~$4$-manifolds. Then, building on this, we define and study the isotropic functor in subsection~\ref{sub:isotr}.

\subsection{Generalized eigenspaces}
\label{sub:eval}

Let~$k_1,\dots,k_\mu$ be positive integers, and let~$G$ denote the finite abelian group~$C_{k_1}\times\dots\times C_{k_\mu}$.
In all this paragraph, we fix a~$\C$-algebra homomorphism
\[
\chi\colon\C[G]\to\C\,.
\]
Note that such a homomorphism is simply given by a character of~$G$, or equivalently, by the choice for~$i=1,\dots,\mu$
of an element~$\omega_i\in S^1$ whose order divides~$k_i$. In other words, it is given by an element~$\omega=(\omega_1,\dots,\omega_\mu)$ of~$\mathbb{T}^\mu$.
Note also that such a~$\chi$ automatically preserves the involutions given
by~$\sum z_g g \mapsto \sum \overline{z_g} g^{-1}$ on~$\C[G]$ and by the complex conjugation on~$\C$.

\begin{terminology}
Given a~$\C[G]$-module~$H$, the {\em generalized eigenspace} associated to the character~$\chi\colon\C[G]\to\C$ is the complex vector space
\[
H_\chi=\lbrace x \in H \ |\ gx=\chi(g)x \,\text{ for all }\, g \in G \rbrace\,.
\]
Since~$H_\chi$ is completely determined by the element~$\omega\in\mathbb{T}^\mu$
corresponding to~$\chi$, we shall often write~$H_\omega$ instead of~$H_\chi$.
\end{terminology}

Denote by~$c_\chi$ the element of~$\C[G]$ defined by 
\[
c_\chi=\frac{1}{|G|}\sum_{g \in G}\overline{\chi(g)}g\,.
\]
One can easily check that for any~$g \in G$, one has~$g\,c_\chi=\chi(g) c_\chi$, which implies that~$c_\chi c_\chi=c_\chi$.
The additional equality~$\chi(c_\chi)=1$ is also easy to check. 
These properties are useful to give an alternative characterization of generalized eigenspaces.

\begin{lemma}
\label{lem: charac}
For any~$\C[G]$-module~$H$, the generalized eigenspace~$H_\chi$ is equal to~$c_\chi H$.
\end{lemma}

\begin{proof}
If~$c_\chi x$ is an arbitrary element of~$c_\chi H$, then~$g\,c_\chi x=\chi(g) c_\chi x$  so~$c_\chi x$ belongs to~$H_\chi$.
Conversely, if~$x$ is an element of~$H_\chi$, then
\[
c_\chi x= \frac{1}{|G|}\sum_{g\in G} \overline{\chi(g)}g x=\frac{1}{|G|}\sum_{g\in G} \overline{\chi(g)} \chi(g) x=x\,,
\]
so~$x=c_\chi x$ lies in~$c_\chi H$.
\end{proof}

If~$H$ and~$H'$ are~$\C[G]$-modules, then any~$\C[G]$-linear map~$f \colon H \rightarrow H'$ restricts to a map~$f_\chi \colon H_\chi \rightarrow H_\chi'$
on generalized eigenspaces, thus defining a functor from the category of~$\C[G]$-modules to the category of complex vector spaces.
Let us analyse some further properties of this functor.

\begin{proposition}
\label{prop: nondeg}
\begin{enumerate}
\item The functor~$H\mapsto H_\chi$ preserves exact sequences.
\item If~$V$ is a submodule of~$H$ satisfying~$c_\chi V=0$, then~$(H/V)_\chi=H_\chi$.
\item If~$\xi$ is a non-degenerate form on~$H$ and~$G$ acts on~$H$ by unitary isomorphisms, then the restriction of~$\xi$
to~$H_\chi$ is also non-degenerate.
\end{enumerate}
\end{proposition}

\begin{proof}
For the first assertion, consider~$\C[G]$-linear maps~$f\colon H \rightarrow H'$ and~$g\colon H' \rightarrow H''$ such that~$\mathit{Ker}(g)=\mathit{Image}(f)$;
we must show that~$\mathit{Ker}(g_\chi)=\mathit{Image}(f_\chi)$. One inclusion follows directly from the functoriality~$(g\circ f)_\chi=g_\chi\circ f_\chi$. For the other
one, fix~$c_\chi x \in \mathit{Ker}(g_\chi) = H'_\chi\cap\mathit{Ker}(g)$. By exactness, there exists~$y \in H$ such that~$f(y)=c_\chi x$. As~$f_\chi(c_\chi y)=c_\chi f(y)=c_\chi c_\chi x=c_\chi x$, the equality is proved. The second statement follows from the first one together with Lemma~\ref{lem: charac} and the hypothesis~$c_\chi V=0$:
\[
(H/V)_\chi=H_\chi/V_\chi=H_\chi/c_\chi V=H_\chi\,.
\]
Finally, let~$c_\chi x$ be an element of~$c_\chi H=H_\chi$ such that~$\xi(c_\chi x, c_\chi y)=0$ for every~$c_\chi y \in H_\chi$.
Using the fact that the elements of~$G$ act by isometries, together with the equality~$c_\chi c_\chi=c_\chi$, one obtains 
\[
0=\xi(c_\chi x, c_\chi y)=\xi(c_\chi c_\chi x, y)=\xi(c_\chi x,y)
\]
for every~$y \in H$. As~$\xi$ is non-degenerate on~$H$, this forces~$c_\chi x=0$.
\end{proof}

As the reader might have guessed, another description of these generalized eigenspaces can be given using tensor products. Indeed, the
homomorphism~$\chi\colon\C[G] \rightarrow \C$ endows the field~$\C$ with a structure of module over~$\C[G]$. To emphasize this action, we shall denote this~$\C[G]$-module by~$\C_\chi$.  Given a Hermitian~$\C[G]$-module~$(H,\xi)$, one can therefore consider the complex vector space~$H\otimes_{\C[G]}\C_\chi$ endowed with the skew-Hermitian form~$\xi^\chi(x \otimes u, y \otimes v)=u \overline{v} \chi(\xi(x,y))$.

\begin{proposition}
\label{prop: eval eingenspace}
Given any Hermitian~$\C[G]$-module~$(H,\xi)$, the map~$\Phi_H\colon H \otimes_{\C[G]} \C_\chi \rightarrow H_\chi$ defined by~$\Phi_H(x \otimes u)= u c_\chi x$ is an
isomorphism of complex vector spaces, unitary with respect to the forms~$\xi^\chi$ and~$\chi\circ\xi$. Furthermore,
~$f_\chi\circ\Phi_H=\Phi_{H'}\circ (f\otimes\mathit{id}_{\C_{\chi}})$ for any~$\C[G]$-linear map~$f\colon H\to H'$.
\end{proposition}

\begin{proof}
The map~$\Phi_H$ is surjective thanks to Lemma~\ref{lem: charac}, while its injectivity follows from the equation
\[
x \otimes 1= x \otimes \chi(c_\chi)= c_\chi x \otimes 1 = \Phi(x\otimes 1)\otimes 1\,.
\]
The equality~$\chi(c_\chi)=1$ easily implies that~$\Phi_H$ is unitary. Finally, the last statement follows from the definitions.
\end{proof}

Note that if~$H$ if a free~$\C[G]$-module of rank~$n$, then~$H_\chi=H \otimes_{\C[G]} \C_\chi$ is a complex vector space of dimension~$n$. By standard properties of the tensor product, we also have the following result, that we record here for further use.

\begin{lemma}
\label{lem: eval}
Let~$f\colon H \rightarrow H'$ be a~$\C[G]$-linear map between free~$\C[G]$-modules and let~$\xi$ be a skew-Hermitian form on~$H$. Fix bases~$v_1,\dots,v_m$ for~$H$
and~$w_1,\dots,w_n$ for~$H'$. Then, the matrix for~$f \otimes \mathit{id}_{\C_{\chi}}$ (resp.~$\xi^\chi$) with respect to the bases~$v_1 \otimes 1,\dots,v_m\otimes 1$
and~$w_1 \otimes 1,\dots,w_n \otimes 1$ is equal to the componentwise evaluation by~$\chi$ of the matrix for~$f$ (resp.~$\xi$).
\end{lemma}

By Proposition~\ref{prop: eval eingenspace}, the same result holds with~$f\otimes \mathit{id}_{\C_{\chi}}$,~$\xi^\chi$ and~$v_i\otimes 1$ replaced
by~$f_\chi$,$\chi\circ\xi$ and~$c_\chi v_i$, respectively.
Finally, note that all the results of this paragraph still hold if we consider Hermitian forms instead of skew-Hermitian ones.

\subsection{Signatures of~$4$-manifolds}
\label{sub:sign2}

The aim of this paragraph is to review the signatures associated to~$4$-manifolds endowed with the action of a finite abelian group. In particular, we shall recall
the celebrated {\em Novikov-Wall theorem\/} on the non-additivity of these signatures.

\medbreak

Let~$M$ be a compact oriented~$2n$-dimensional manifold endowed with the action of a finite abelian group~$G$. The homology of~$M$ with complex coefficients is endowed with a structure of module over~$\C[G]$. In particular, if~$\chi\colon \C[G] \rightarrow \C$ is a ~$\C$-algebra homomorphism, one may consider the generalized eigenspace 
\[
H_n(M)_\chi:=\lbrace x \in H_n(M;\C) \ | \ \chi(g)x=gx \ \text{for all} \ g \in G \rbrace\,.
\]
This complex vector space comes equipped with a~$(-1)^n$-Hermitian form given by the restriction of the intersection form~$\langle \ , \ \rangle$ of~$H_n(M;\C)$.
On the other hand, for any~$x,y \in H_n(M;\C)$, one can define
\[
\xi(x,y)=\sum_{g \in G} \langle gx,y\rangle g^{-1}\,,
\] 
and~$\chi\circ\xi$ gives another complex valued~$(-1)^n$-Hermitian form on~$H_n(M)_\chi $. It turns out that these two forms are closely related.

\begin{proposition}
\label{prop: forms}
On the space~$H_n(M)_\chi$, the intersection form~$\langle \ , \ \rangle$ and the pairing~$\chi \circ \xi $ coincide up to a positive multiplicative constant.
\end{proposition}

\begin{proof}
Fix arbitrary elements~$c_\chi x, c_\chi y \in H_n(M)_\chi=c_\chi H_n(M;\C)$. Using the fact that~$\chi$ is a ring homomorphism, the definition of~$c_\chi$ and the
equality~$c_\chi c_\chi=c_\chi$, we get
\begin{align*}
\chi(\xi(c_\chi x, c_\chi y))&=\sum_{g\in G} \langle g c_\chi x, c_\chi y\rangle \overline{\chi(g)}= \langle \sum_{g\in G}\overline{\chi(g)} g c_\chi x,c_\chi y\rangle\\
	&=|G|\langle c_\chi c_\chi x, c_\chi y \rangle=|G|\langle c_\chi x, c_\chi y\rangle\,,
\end{align*}
and the proposition is proved.
\end{proof}

Proposition~\ref{prop: forms}, Proposition~\ref{prop: eval eingenspace} and Lemma~\ref{lem: eval} immediately yield the following corollary.

\begin{corollary}
\label{cor:eval}
Up to a positive multiplicative constant, a matrix of the restriction of the intersection form to~$H_n(M)_\chi$ is given by a matrix of the form~$\xi$
evaluated componentwise by~$\chi$.
\end{corollary}

Note that when~$n$ is even, these two forms are Hermitian and therefore have a well-defined (identical) signature. It is called the {\em $\chi$-signature} of~$M$,
and will be denoted by~$\sigma_\chi(M)$.
Since it is completely determined by an element~$\omega$ of~$\mathbb{T}^\mu$, we shall sometimes write~$\sigma_\omega(M)$ instead of~$\sigma_\chi(M)$, and call it
the~{\em $\omega$-signature} of~$M$.

The (non-)additivity of this signature is well-understood thanks to a famous theorem of C.T.C Wall~\cite{CTC}.
(Strictly speaking, Wall only stated and proved his result for ordinary signatures of manifolds; however, he did mention in~\cite[p.274]{CTC} that
it extends to~$G$-manifolds and~$G$-signatures.) This result holds for any even~$n$, but we shall restrict
ourselves to low-dimensional manifolds. To state this theorem, we need the following well-known consequence of Poincar\'e duality.

\begin{lemma}
\label{lem: duality}
Let~$X$ be a compact oriented~$3$-dimensional manifold-with-boundary endowed with the action of a finite abelian group~$G$.
Then, its boundary~$\Sigma=\partial X$ inherits an orientation and a~$G$-action from~$X$, and
the kernel of the map induced by the inclusion of~$\Sigma$ in~$X$ is a Lagrangian subspace of~$H_1(\Sigma)_\chi$ with respect to the intersection form.
\end{lemma}

Let~$M$ be an oriented compact~$4$-manifold endowed with the action of a finite abelian group~$G$ and let~$X_0$ be an oriented compact~$3$-manifold properly embedded
into~$M$, so that~$X_0$ intersects~$\partial M$ along~$\partial X_0=X_0 \cap \partial M$. Assume that~$X_0$ splits~$M$ into two manifolds~$M_1$ and~$M_2$. For~$i=1,2$, denote by~$X_i$ the compact~$3$-manifold~$\partial M_i\setminus\mathit{Int}(X_0)$. Orient~$X_1$ and~$X_2$ so that~$\partial M_1= X_0 \cup (-X_1)$
and~$\partial M_2=(-X_0) \cup X_2$. Note that the orientations of~$X_0$,~$X_1$ and~$X_2$ induce the same orientation on the
surface~$\Sigma=\partial X_0=\partial X_1 = \partial X_2$. By Lemma~\ref{lem: duality}, we know that given any~$\C$-algebra homomorphism~$\chi\colon\C[G]\rightarrow\C$,
the subspace~$(L_i)_\chi=\mathit{Ker}(H_1(\Sigma)_\chi \rightarrow H_1(X_i)_\chi)$ is Lagrangian in~$H_1(\Sigma)_\chi$ for~$i=0,1,2$.

\begin{theorem}[Wall~\cite{CTC}]
\label{thm: Wall}
Under the conditions above, the~$\chi$-signature of~$M$ is given by
\[
\sigma_\chi(M)=\sigma_\chi(M_1)+\sigma_\chi(M_2)+\mathit{Maslov}((L_1)_\chi,(L_0)_\chi,(L_2)_\chi).
\]
\end{theorem}

As in~\cite{GG}, this result will be one of the main tools in the proof of our result. We will often refer to it as the {\em Novikov-Wall theorem\/}.

\subsection{The isotropic functor}
\label{sub:isotr}

Given a fixed torsion element~$\omega$ in~$\mathbb{T}^\mu$, the aim of this subsection is now to define a functor~$\mathcal{F}_\omega\colon\mathbf{Tangles}_\mu \rightarrow\mathbf{Isotr}_{\C}$
using branched coverings. We will then give a sufficient condition for this functor to take its values in the Lagrangian category~$\mathbf{Lagr}_\C$.

\medbreak

Given a positive integer~$n$, recall that~$x_j$ denotes the point~$((2j-n-1)/n,0)$ in the closed unit disk~$D^2$, for~$j=1,\dots,n$.
Let~$\mathcal{N}(\lbrace x_1,\dots,x_n \rbrace)$ be an open tubular neighborhood of~$\lbrace x_1,\dots,x_n \rbrace$.
Given a map~$c\colon\lbrace 1,\dots,n \rbrace \rightarrow \lbrace \pm1,\dots,\pm \mu \rbrace$, we shall denote by~$D_c$ the compact surface
\[
D_c=D^2\setminus\mathcal{N}(\lbrace x_1,\dots,x_n \rbrace)
\]
endowed with the counterclockwise orientation and a basepoint~$z$. The same space with the clockwise orientation will be denoted by~$-D_c$. The fundamental group~$\pi_1(D_c,z)$ is freely generated by~$\lbrace e_1,\dots,e_n \rbrace$, where~$e_j$ is a simple loop turning once around~$x_j$ counterclockwise if~$\text{sgn}(c_j)=1$, clockwise if~$\text{sgn}(c_j)=-1$. 

Fix a torsion element~$\omega=(\omega_1,\dots,\omega_\mu)\in\mathbb{T}^\mu$, let~$k_i$ be the order of~$\omega_i$ and~$G$ be the finite abelian
group~$C_{k_1} \times\dots\times C_{k_\mu}$. Also, let~$C^\mu_\infty$ be the (multiplicative) free abelian group with basis~$t_1,\dots,t_\mu$.
Composing the coloring induced homomorphism~$H_1(D_c) \rightarrow C_\infty^\mu, \ e_j \mapsto t_{|c_j|}$ with the canonical projection~$C_\infty^\mu \rightarrow G$ yields a regular
$G$-covering~$\widehat{D}_c \rightarrow D^2$ branched along the punctures. The homology group~$H_1(\widehat{D}_c;\C) $ is endowed with a structure of module over~$\C[G]$.
Let~$\langle \ , \ \rangle_c\colon H_1(\widehat{D}_c;\C) \times H_1(\widehat{D}_c;\C) \rightarrow \C $ be the (skew-Hermitian) intersection form obtained by lifting the orientation of~$D_c$
to~$\widehat{D}_c$. Restricting this form to the generalized eigenspace~$H_1(\widehat{D}_c)_\omega$ (recall subsection~\ref{sub:eval}) turns the latter into a Hermitian complex vector space.

Given a colored tangle~$\tau \in T_\mu(c,c')$ with~$m$ components, denote by~$\mathcal{N}(\tau)$ an open tubular neighborhood of~$\tau$, and let
\[
X_\tau=(D^2 \times [0,1])\setminus\mathcal{N}(\tau)
\]
be the exterior of~$\tau$. We shall orient~$X_\tau$ so that the induced orientation on~$\partial X_\tau$ extends the orientation on~$(-D_c) \sqcup D_{c'}$.
The long exact sequence of the pair~$(D^2\times [0,1],X_\tau)$, excision and duality yield~$H_1(X_\tau)=\bigoplus_{j=1}^m \Z m_j$,
where~$m_j$ is a meridian linking once the~$j^{\mathit{th}}$ component of~$\tau$.

Composing the coloring induced homomorphism~$H_1(X_\tau) \rightarrow C_\infty^\mu, \ m_j \mapsto t_{|c_j|}$ with the canonical projection~$C_\infty^\mu \rightarrow G$
yields a regular~$G$-covering~$p\colon \widehat{X}_\tau \rightarrow D^2 \times [0,1]$ branched along~$\tau$.
Let~$\hat{\imath}_\tau\colon H_1(\widehat{D}_c)_{\omega} \rightarrow H_1(\widehat{X}_\tau)_{\omega}$ and~$\hat{\imath}_\tau'\colon H_1(\widehat{D}_{c'})_{\omega} \rightarrow H_1(\widehat{X}_\tau)_{\omega}$ be the homomorphisms induced by the inclusions of~$\widehat{D}_c$ and~$\widehat{D}_{c'}$ in~$\widehat{X}_\tau$.
Finally, let~$\hat{\jmath}_\tau$ be the homomorphism
\[
\hat{\jmath}_\tau\colon H_1(\widehat{D}_c)_{\omega}\oplus H_1(\widehat{D}_{c'})_{\omega} \rightarrow H_1(\widehat{X}_\tau)_{\omega}
\]
given by~$\hat{\jmath}_\tau(x,x')=\hat{\imath}_\tau'(x')-\hat{\imath}_\tau(x)$.

\begin{theorem}
\label{thm:isotropicfunctor}
Fix a torsion element~$\omega$ in~$\mathbb{T}^\mu$. Let~$\mathcal{F}_\omega$ assign to each map~$c\colon\lbrace 1,\dots,n \rbrace \rightarrow \lbrace\pm 1,\dots,\pm \mu \rbrace$
the pair~$(H_1(\widehat{D}_c)_{\omega}, \langle \ , \ \rangle_c)$ and to each tangle~$\tau \in T_\mu(c,c')$ the subspace~$\mathit{Ker}(\hat{\jmath}_\tau)$
of~$H_1(\widehat{D}_c)_{\omega} \oplus H_1(\widehat{D}_{c'})_{\omega}$. Then~$\mathcal{F}_\omega $ is a functor~$\mathbf{Tangles}_\mu \rightarrow \mathbf{Isotr}_{\C}$ which fits in the commutative diagram
\[
\begin{xy}
(0,15)*+{\mathbf{Braids}_\mu}="a";
(30,15)*+{\mathbf{Tangles}_\mu}="c";
{\ar@{->}^-{} "a";"c"};
(0,0)*+{\mathbf{\widetilde{U}}_\C}="g";
(30,0)*+{\mathbf{Isotr}_\C\,,}="h";
{\ar@{->}^-{\Gamma} "g";"h"};
{\ar@{->}^{} "a";"g"};
{\ar@{->}^{\mathcal{F}_\omega} "c";"h"};
\end{xy} 
\]
where the horizontal arrows are the embeddings of categories described in Sections~\ref{sub:tangle} and~\ref{sub:cat}. 
\end{theorem}

\begin{proof}
By the discussion above, for any object~$c$ of~$\mathbf{Tangles}_\mu$,~$\mathcal{F}_\omega(c)$ is a Hermitian complex vector space,
i.e. an object of~$\mathbf{Isotr}_{\C}$.

Now, fix a~$\mu$-colored tangle~$\tau\in T_\mu(c,c')$ and let us check that~$\mathcal{F}_\omega(\tau)$ is an isotropic subspace of~$(-H_1(\widehat{D}_c)_{\omega})\oplus H_1(\widehat{D}_c)_{\omega}$. During this discussion, we shall denote by~$\xi$ (respectively~$\xi'$) the skew-Hermitian intersection form on~$H_1(\widehat{D}_c)_{\omega}$
(respectively~$H_1(\widehat{D}_{c'})_{\omega}$).
Recall that we oriented~$X_\tau$ so that the orientation of~$\partial X_\tau $ extends the one of~$(-D_c)\sqcup D_{c'}$. Consequently,
the composition of the form~$\Omega$ on~$H_1(\partial \widehat{X}_\tau)_{\omega}$ with the homomorphism induced by the
inclusion~$(-\widehat{D}_c)\sqcup\widehat{D}_{c'}\subset\partial \widehat{X}_\tau$ is equal to~$(-\xi)\oplus\xi'$. Observe that the map~$\hat{\jmath}_\tau$ is given by the composition 
\[
H_1(\widehat{D}_c)_{\omega}\oplus H_1(\widehat{D}_{c'})_{\omega}\stackrel{\psi}{\longrightarrow}
H_1(\widehat{D}_c)_{\omega}\oplus H_1(\widehat{D}_{c'})_{\omega}\stackrel{i}{\longrightarrow}
H_1(\partial \widehat{X}_\tau)_{\omega}\stackrel{\varphi}{\longrightarrow}H_1(\widehat{X}_\tau)_{\omega}\,,
\]
where~$\psi=(-\mathit{id})\oplus\mathit{id}$ while~$i$ and~$\varphi$ are the inclusion induced maps. Writing~$L:=\mathit{Ker}(\varphi\circ i)$, we
find~$\mathit{Ker}(\hat{\jmath}_\tau)=\psi(L)$ and consequently~$\mathrm{Ann}(\mathit{Ker}(\hat{\jmath}_\tau))=\mathrm{Ann}((\psi(L))=\psi(\mathrm{Ann}(L))$.
Let us now check that~$L$ is isotropic. Given~$x,y \in L$, the elements~$i(x)$ and~$i(y)$ belong to~$\mathit{Ker}(\varphi)$ which is known to be Lagrangian by Lemma~\ref{lem: duality}.
As the form~$\Omega$ ``restricts to''~$(-\xi) \oplus \xi' $ on~$H_1((-\widehat{D}_c) \sqcup \widehat{D}_{c'})_{\omega}$, we get
\[
((-\xi) \oplus \xi' )(x,y)= \Omega(i(x),i(y)) =0
\]
as desired. Combining these observations, it follows that
\[
\mathit{Ker}(\hat{\jmath}_\tau)=\psi(L) \subset \psi(\mathrm{Ann}(L))=\mathrm{Ann}(\mathit{Ker}(\hat{\jmath}_\tau))\,,
\]
which shows that~$\mathcal{F}_\omega(\tau)=\mathit{Ker}(\hat{\jmath}_\tau)$ is isotropic. 

The proof of the functoriality follows by restricting the arguments given in~\cite[Lemma 3.4]{CT} to generalized eigenspaces. (Recall that the first point of
Proposition~\ref{prop: nondeg} ensures that exactness is preserved.) Finally, the restriction of this functor to braids
can be analysed by a straightforward adaptation of the proof of~\cite[Proposition 5.1]{CT}.
\end{proof}

Our next goal is to find a sufficient condition for the functor~$\mathcal{F}_\omega$ to take its values in the Lagrangian category~$\mathbf{Lagr}_\C$.
Given~$\omega\in\mathbb{T}^\mu$, let~$\mathbf{Tangles}_\mu^\omega$ be the full subcategory of~$\mathbf{Tangles}_\mu$ whose objects are
maps~$c\colon\lbrace 1,\dots,n \rbrace \rightarrow \lbrace \pm 1,\dots, \pm \mu \rbrace$ such that~$\ell(c)$ is nowhere zero
and such that~$\omega$ belongs to~$\mathbb{T}^\mu_c$ (recall subsection~\ref{sub:thm}).
We shall denote by~$T_\mu^\omega(c,c')$ the set of morphisms between two objects~$c$ and~$c'$ of this category.

\begin{lemma}
\label{lemma: boundary}
Fix~$\omega$ in~$\mathbb{T}^\mu_P$. If~$c$ is an object of~$\mathbf{Tangles}_\mu^\omega$, then the surface~$\widehat{D}_c $ has one boundary component,
and the restriction of the skew-Hermitian intersection form to~$H_1(\widehat{D}_c)_{\omega}$ is non-degenerate.
\end{lemma}

\begin{proof}
As the branch set of the covering is contained in the interior of the disk~$D^2$, the composition~$\pi_1(\partial D^2)\to\pi_1(D_c)\to G$
induces a regular~$G$-covering. Consequently, the boundary of the total space~$\widehat{D}_c$ has one component if and only if this homomorphism is surjective.
In other words, the connectedness of~$\partial\widehat{D}_c$ is equivalent to the image of the generator of~$\pi_1(\partial D^2)$ spanning~$G$. The image of the generator
of~$\pi_1(\partial D^2)$ in~$\pi_1(D_c)$ goes once around each of the punctures, so it is sent by the above composition to the class of~$\ell(c)$ in~$G$.
Using the Chinese remainder theorem, this element generates~$G$ if and only if the~$k_i>1$'s are pairwise coprime (i.e.~$\omega$ belongs to~$\mathbb{T}^\mu_P$), 
and~$\ell(c)_i$ and~$k_i$ are coprime for each~$i$ (i.e. $c$ is an object of~$\mathbf{Tangles}_\mu^\omega$). In this case,~$\widehat{D}_c$ has one boundary component, so the intersection form on~$H_1(\widehat{D}_c;\C)$ is non-degenerate. The last claim follows
from the third part of Proposition~\ref{prop: nondeg}.
\end{proof}

Lemma~\ref{lemma: boundary} gives a sufficient condition for the functor~$\mathcal{F}_\omega$ to take its values in the Lagrangian category~$\mathbf{Lagr}_\C$.

\begin{proposition}
\label{prop:FLagr}
Fix an element~$\omega$ in~$\mathbb{T}_P^\mu$ (i.e. assume that the component~$\omega_i$ of~$\omega$ is of order~$k_i>1$ with these~$k_i$'s pairwise coprime).
Then the restriction of~$\mathcal{F}_\omega$ to~$\mathbf{Tangles}_\mu^\omega$ defines a functor which fits in the commutative diagram 
\[
\begin{xy}
(0,15)*+{\mathbf{Braids}^\omega_\mu}="a";
(30,15)*+{\mathbf{Tangles}^\omega_\mu}="c";
{\ar@{->}^-{} "a";"c"};
(0,0)*+{\mathbf{U}_\C}="g";
(30,0)*+{\mathbf{Lagr}_\C.}="h";
{\ar@{->}^-{\Gamma} "g";"h"};
{\ar@{->}^{} "a";"g"};
{\ar@{->}^{\mathcal{F}_\omega} "c";"h"};
\end{xy} 
\]
\end{proposition}

\begin{proof}
Let~$\tau \in T_\mu^\omega(c,c')$ be a colored tangle. Applying the same notation and reasoning as in the proof of Theorem~\ref{thm:isotropicfunctor}, we only need to show
that~$\mathit{Ker}(\varphi\circ i)$ is Lagrangian. Since~$\mathit{Ker}(\varphi)$ is Lagrangian by Lemma~\ref{lem: duality}, we are left with the proof that~$i$ is an isomorphism.
To check this claim, note that
\[
\partial \widehat{X}_\tau= (\widehat{D}_c \sqcup \widehat{D}_{c'}) \cup p^{-1}(S^1 \times [0,1])\,,
\]
where~$p^{-1}(S^1 \times [0,1])$ consists of a  certain number of disjoint cylinders. As~$\omega$ belongs to~$\mathbb{T}^\mu_P$ and~$\tau$ to~$T_\mu^\omega(c,c')$,
Lemma~\ref{lemma: boundary} implies that both~$\widehat{D}_c$ and~$\widehat{D}_{c'} $ have a single boundary component. Consequently,~$p^{-1}(S^1 \times [0,1])$ consists of a single cylinder, so the inclusion~$\widehat{D}_c\sqcup \widehat{D}_{c'}\hookrightarrow\widehat{X}_\tau$ induces an isomorphism on the first homology groups, as claimed. 
\end{proof}

\begin{remark}
\label{rem: Fcharac}
Fix an element~$\omega$ in~$\mathbb{T}_P^\mu$ and a tangle~$\tau\in T_\mu^\omega(c,c')$. The proof of the previous proposition shows that
\[
\mathcal{F}_\omega(\tau) = \psi(\mathit{Ker}(j))\,,
\]
where~$\psi$ is the unitary automorphism of~$H_1(\widehat{D}_c)_{\omega} \oplus H_1(\widehat{D}_{c'})_{\omega}$ given by~$\psi=(-\mathit{id})\oplus\mathit{id}$
while~$j$ is the inclusion induced homomorphism
\[
j\colon H_1(\widehat{D}_c)_{\omega} \oplus H_1(\widehat{D}_{c'})_{\omega}\cong H_1(\partial \widehat{X}_\tau)_{\omega}\to H_1(\widehat{X}_\tau)_{\omega}\,.
\]
This characterization will help us later on for using the Novikov-Wall theorem.
\end{remark}

We conclude this section with one last property of this functor, namely the fact that it behaves well with respect to juxtaposition of colored tangles.

\begin{proposition}
\label{prop: Fdisjoint}
For any~$\omega \in \mathbb{T}_P^\mu$,~$\tau_1 \in T^\omega_\mu(c_1,c_1')$ and~$\tau_2 \in T^\omega_\mu(c_2,c_2')$, we have
\begin{enumerate}
\itemsep0em
\item $H_1(\widehat{D}_{c_1 \sqcup c_2})_{\omega} \cong H_1(\widehat{D}_{c_1})_{\omega} \oplus H_1(\widehat{D}_{c_2})_{\omega} \oplus A$ as Hermitian complex vector spaces,
\item $\mathcal{F}_\omega(\tau_1 \sqcup \tau_2) \cong \mathcal{F}_\omega(\tau_1) \oplus \mathcal{F}_\omega(\tau_2) \oplus \Delta_A$,
\end{enumerate}
where~$A$ is some subspace of~$H_1(\widehat{D}_{c_1 \sqcup c_2})_{\omega}$ and~$\Delta_A=\lbrace x \oplus x \ | \ x \in A \rbrace$ the associated diagonal.
\end{proposition}

\begin{proof}
The space~$D_{c_1 \sqcup c_2}$ can be obtained by gluing~$D_{c_1}$ and~$D_{c_2}$ along an interval~$I$ in their boundary. As intervals are contractible, this decomposition lifts 
to~$\widehat{D}_{c_1 \sqcup c_2}=\widehat{D}_{c_1} \cup_{I \times G} \widehat{D}_{c_2}$. 
Applying the same line of reasoning to the tangle exteriors~$X_{\tau_1}$ and~$X_{\tau_2}$ and using the corresponding Mayer-Vietoris exact sequence leads to the commutative diagram
\[
\begin{xy}
(0,15)*+{0}="a";
(30,15)*+{H_1(\widehat{D}_{c_1})_{\omega}\oplus H_1(\widehat{D}_{c_2})_{\omega}}="b";
(70,15)*+{H_1(\widehat{D}_{c_1 \sqcup c_2})_{\omega}}="c";
(100,15)*+{\mathit{Image}(\partial)}="d";
(120,15)*+{0}="e";
{\ar@{->}^-{} "a";"b"};
{\ar@{->}^-{} "b";"c"};
{\ar@{->}^-{\partial} "c";"d"};
{\ar@{->}^-{} "d";"e"};
(0,0)*+{0 }="f";
(30,0)*+{H_1(\widehat{X}_{\tau_1})_{\omega} \oplus H_1(\widehat{X}_{\tau_2})_{\omega}}="g";
(70,0)*+{H_1(\widehat{X}_{\tau_1 \sqcup \tau_2})_{\omega}}="h";
(100,0)*+{\mathit{Image}(\partial)}="i";
(120,0)*+{0.}="j";
{\ar@{->}^-{} "f";"g"};
{\ar@{->}^-{} "g";"h"};
{\ar@{->}^-{\partial} "h";"i"};
{\ar@{->}^-{} "i";"j"};
{\ar@{->}^{\hat{\imath}_{\tau_1} \oplus \hat{\imath}_{\tau_2}} "b";"g"};
{\ar@{->}^{\hat{\imath}_{\tau_1 \sqcup \tau_2}} "c";"h"};
{\ar@{=} "d";"i"};
\end{xy} 
\]
Splitting these short exact sequences of vector spaces and writing~$A=\mathit{Image}(\partial)$, one gets the decompositions
\[
H_1(\widehat{D}_{c_1 \sqcup c_2})_{\omega} \cong H_1(\widehat{D}_{c_1})_{\omega} \oplus H_1(\widehat{D}_{c_2})_{\omega} \oplus A
\]
and
\[
H_1(\widehat{X}_{\tau_1 \sqcup \tau_2})_{\omega} \cong H_1(\widehat{X}_{\tau_1})_{\omega} \oplus H_1(\widehat{X}_{\tau_2})_{\omega} \oplus A\,.
\]
At the level of maps, one obtains~$\hat{\imath}_{\tau_1 \sqcup \tau_2}=\hat{\imath}_{\tau_1} \oplus \hat{\imath}_{\tau_2} \oplus \mathit{id}_A$, so the vector
space~$\mathcal{F}_\omega(\tau_1 \sqcup \tau_2)$ is isomorphic to~$\mathit{Ker}(\hat{\jmath}_{\tau_1}) \oplus \mathit{Ker}(\hat{\jmath}_{\tau_2}) \oplus \Delta_A$ as claimed.
To conclude the proof, we still need to check that the first decomposition displayed above is orthogonal with respect to the intersection forms.
As~$\omega$ belongs to~$\mathbb{T}^\mu_P$ and the tangles are morphisms of~$\mathbf{Tangles}_\mu^\omega$, Lemma~\ref{lemma: boundary} implies
that~$\widehat{D}_{c_1}$ and~$\widehat{D}_{c_2}$ are compact surfaces with one boundary component. It follows that the section of the exact sequence above can be chosen so that
the corresponding decomposition is orthogonal.
\end{proof}


\section{Proof of Theorem~\ref{thm:main}}
\label{sec:proof}

This section is devoted to the proof of our main result, a proof which extends (and hopefully, at times, also clarifies) the one of~\cite{GG}.
Let us very briefly outline the underlying strategy. We will build~$4$-manifolds whose~$\omega$-signatures are equal to the terms appearing in the theorem.
Gluing these manifolds together yields a manifold whose~$\omega$-signature is equal to
\[
\sign_\omega(\widehat{\tau_1 \tau_2})-\sign_\omega(\widehat{\tau_1})-\sign_\omega(\widehat{\tau_2}) -\mathit{Maslov}(\mathcal{F}_\omega(\overline{\tau}_1),\Delta,\mathcal{F}_\omega(\tau_2))\,.
\]
It will then only remain to show that this signature vanishes.

Actually, it is sufficient to show all these equalities up to a uniformly bounded constant, thanks to a reduction of our main theorem to a looser statement.
This is the subject of the first subsection.

\subsection{A reduction}
\label{sub:red}

If~$n$ and~$m$ are two integers depending on some tangles, we shall write~$ n \simeq m $ if~$|n-m|$ is bounded by a constant that is independent of the tangles.
The aim of this paragraph is to prove the following proposition which will spare us the trouble of keeping track of (most of) the Novikov-Wall defects.

\begin{proposition}
\label{prop: reduction}
To prove Theorem~\ref{thm:main}, it is enough to show that for any~$c$ such that~$\ell(c)$ is nowhere zero and for any~$(c,c)$-tangles~$\tau_1$ and~$\tau_2$,
we have
\[
\mathit{sign}_\omega(\widehat{\tau_1\tau_2})-\mathit{sign}_\omega(\widehat{\tau_1})-\sign_\omega(\widehat{\tau_2}) \simeq \mathit{Maslov}(\mathcal{F}_\omega( \overline{\tau}_1), \Delta ,\mathcal{F}_\omega(\tau_2))
\]
for all~$\omega$ in~$\mathbb{T}^\mu_c\cap\mathbb{T}^\mu_P$.
\end{proposition}

We know from Lemma~\ref{lemma:sign} that the signature defect (i.e. the left-hand side of the equation displayed above) is additive with respect to the disjoint union of tangles.
Furthermore, Lemma~\ref{lemma:Maslov} and Proposition~\ref{prop: Fdisjoint} immediately imply that the right-hand side of this equation, that we shall denote by~$M_\omega(\tau_1,\tau_2)$,
satisfies
\[
M_\omega(\tau_1\sqcup\tau_1',\tau_2\sqcup\tau_2')=M_\omega(\tau_1,\tau_2)+M_\omega(\tau'_1,\tau'_2)
\]
if~$\omega$ belongs to~$\mathbb{T}^\mu_P$ and the tangles~$\tau_1,\tau_2,\tau_1',\tau_2'$ are morphisms of the category~$\mathbf{Tangles}^\omega_\mu$.
This will be the key ingredient in the proof of the reduction. We will also need the following easy lemma.

\begin{lemma}
\label{lem:gcd}
Given two coprime integers~$\ell\neq 0$ and~$k>0$, there exist two positive integers~$m$ and~$n$ such that~$m+n+\ell$, $2\ell+m$ and~$n$ are positive and coprime to~$k$.
\end{lemma}

\begin{proof}
Set~$m=\lambda k-\ell$ and~$n=\lambda k+\ell$ for any integer~$\lambda>0$ such that~$m$ and~$n$ are positive.
\end{proof}

Proposition~\ref{prop: reduction} will be an easy consequence of the following statement.

\begin{lemma}
\label{lem:trick}
Let~$\omega$ be an element of~$\mathbb{T}_P^\mu$. For any~$\tau_1, \tau_2 \in T_\mu^\omega(c,c)$, there exists a 
coloring~$c'$ that is an object of the category~$\mathbf{Tangles}_\mu^\omega$ and~$(c',c')$-colored tangles~$\tau_1',\tau_2'$ such
that~$M_\omega(\tau_1',\tau_2')=2M_\omega(\tau_1,\tau_2)$.
\end{lemma}

\begin{proof}
Since~$c$ is an object of~$\mathbf{Tangles}_\mu^\omega$, we can apply Lemma~\ref{lem:gcd} to~$\ell_i=\ell(c)_i\neq 0$ and~$k_i>0$, thus
producing positive integers~$m_i$ and~$n_i$ for~$i=1,\dots,\mu$. Set
\[
\textstyle \tau_1':=\tau_1 \sqcup \tau_1 \sqcup \bigsqcup_{i=1}^\mu m_i\quad\text{and}\quad  \tau_2':=\tau_2 \sqcup \tau_2\sqcup \bigsqcup_{i=1}^\mu m_i\,,
\]
where~$m_i$ denotes the trivial tangle with~$m_i$ (upward oriented) strands of color~$i$, and let~$c'$ be the corresponding coloring.
The fact that~$c'$ is an object of~$\mathbf{Tangles}_\mu^\omega$ follows from the second point of~Lemma~\ref{lem:gcd}.
By the second and third points of this same lemma, we have
\[
\textstyle M_\omega(\tau_1',\tau_2')
= M_\omega \left( \tau_1 \sqcup \tau_1 \sqcup \bigsqcup_{i=1}^\mu (m_i \sqcup n_i),\tau_2\sqcup \tau_2 \sqcup \bigsqcup_{i=1}^\mu (m_i \sqcup n_i) \right)\,,
\]
which splits as
\[
\textstyle M_\omega(\tau_1,\tau_2)+ M_\omega \left( \tau_1 \sqcup \bigsqcup_{i=1}^\mu (m_i \sqcup n_i ) ,\tau_2 \sqcup \bigsqcup_{i=1}^\mu (m_i \sqcup n_i  ) \right)
\]
by the first point of Lemma~\ref{lem:gcd} and the fact that~$\tau$ is a morphism of~$\mathbf{Tangles}_\mu^\omega$.
Finally, adding the trivial tangle~$\mathit{id}_c$ and using twice more a combination of the first part of Lemma~\ref{lem:gcd} and the fact that~$\tau$ is a
morphism~$\mathbf{Tangles}_\mu^\omega$, we obtain
\[
\textstyle M_\omega(\tau_1',\tau_2')=2M_\omega(\tau_1,\tau_2)+M_\omega \left( \bigsqcup_{i=1}^\mu (m_i \sqcup n_i \sqcup \ell_i) , \bigsqcup_{i=1}^\mu (m_i \sqcup n_i \sqcup \ell_i) \right)= 2M_\omega(\tau_1,\tau_2)\,,
\]
which concludes the proof.
\end{proof}

\begin{proof}[Proof of Proposition~\ref{prop: reduction}]
Assume by contradiction that for a fixed map~$c$, there are~$(c,c)$-tangles~$\tau_1,\tau_2$ and an element~$\omega$ of~$\mathbb{T}^\mu_c \cap \mathbb{T}^\mu_P$ such that
\[
N_\omega(\tau_1,\tau_2):=\sign_\omega(\widehat{\tau_1 \tau_2})-\sign_\omega(\widehat{\tau_1})-\sign_\omega(\widehat{\tau_2}) -\mathit{Maslov}(\mathcal{F}_\omega(\overline{\tau}_1),\Delta, \mathcal{F}_{\omega}(\tau_2))
\]
does not vanish. For any positive integer~$m$, using inductively Lemma~\ref{lem:trick}, one obtains a coloring~$c(m)$ and tangles~$\tau_1(m),\tau_2(m)$ such that~$\omega$
belongs to~$T_{c(m)}\cap\mathbb{T}^\mu_P$ and 
\[
N_\omega(\tau_1(m),\tau_2(m))=2^m\,N_\omega(\tau_1,\tau_2)\,.
\]
Since this quantity goes to infinity as~$m$ grows, this concludes the proof.
\end{proof}

\begin{remark}
\label{rem:GG}
The idea of this reduction comes from the paper~\cite{GG} of Gambaudo and Ghys. Let us mention however that these authors use a simpler
version of this trick, which turns out to be slightly incorrect. (See the last line of~\cite[p.559]{GG}, where it is claimed that the reduced Burau
representation evaluated at~$t=-1$ is always additive
under disjoint union.) To the best of our knowledge, the more involved trick given above seems to be needed even in the case~$\mu=1$.
\end{remark}

\subsection{The manifold~$P_G(\tau_1,\tau_2)$.}
\label{sub:P}

Fix a map~$c\colon\lbrace  1,\dots, n \rbrace \rightarrow \lbrace \pm 1,\dots, \pm\mu \rbrace $ with~$\ell(c)$ nowhere zero, an element~$\omega$
in~$\mathbb{T}_c^\mu\cap\mathbb{T}_P^\mu$ and two~$(c,c)$-tangles~$\tau_1$,~$\tau_2$. Denote by~$G$ the finite abelian group~$C_{k_1} \times \dots\times C_{k_\mu}$,
where~$k_i$ is the order of~$\omega_i$.

Let~$P$ be the sphere with three holes more commonly known as a ``pair of pants'', with a fixed orientation that will be pictured as counterclockwise.
Let~$I_1$ and~$I_2$ be closed intervals joining the inner boundary components of the pair of pants to the outer boundary component. Thicken these intervals
to get rectangles~$J_1=I_1 \times [0,1]$ and~$J_2=I_2 \times [0,1]$, as illustrated in Figure~\ref{fig:pants}.

\begin{figure}[tb]
\labellist\small\hair 2.5pt
\pinlabel {$J_1$} at 113 512
\pinlabel {$J_2$} at 460 512
\pinlabel {$\gamma$} at 265 657
\pinlabel {$P$} at 90 690
\pinlabel {$I_1\times\{0\}$} at 118 545
\pinlabel {$I_1\times\{1\}$} at 118 477
\pinlabel {$I_2\times\{0\}$} at 458 477
\pinlabel {$I_2\times\{1\}$} at 458 545
\endlabellist
\centering
\includegraphics[width=0.4\textwidth]{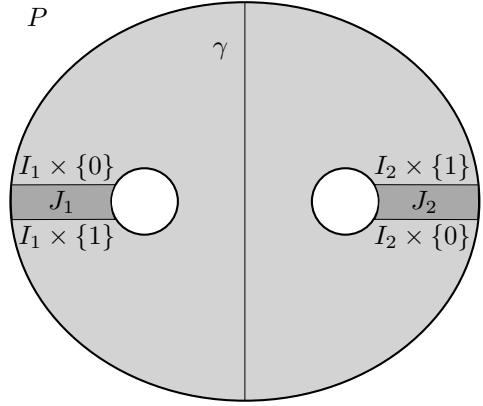}
\caption{The various decompositions of the pair of pants~$P$.}
\label{fig:pants}
\end{figure}

Define~$R(\tau_1,\tau_2)$ as the surface in~$P\times D^2$ which coincides with
\begin{enumerate}
\itemsep0em 
\item the surface~$(P\setminus(J_1 \cup J_2)) \times \lbrace x_1,\dots,x_n \rbrace$ on~$(P\setminus(J_1 \cup J_2))\times D^2$,
\item the surface~$I_1 \times \tau_1,$ on~$ J_1 \times D^2=I_1\times ([0,1] \times D^2)$, and
\item the surface~$I_2 \times \tau_2,$ on~$ J_2 \times D^2=I_2\times ([0,1] \times D^2)$.
\end{enumerate}
Observe that for each point~$x\in I_i$ ($i=1,2$), the surface~$R(\tau_1,\tau_2)$ contains a copy of~$\tau_i$;
therefore, its complement~$P \times D^2\setminus R(\tau_1,\tau_2))$ contains one copy of the tangle exterior~$X_{\tau_i}$ for each point in~$I_i$.
Recall from subsection~\ref{sub:isotr} that for any~$(c,c)$-tangle~$\tau$, there is a natural map~$H_1(X_{\tau})\rightarrow G$ obtained by composing the colored-induced
map with the canonical projection.

\begin{lemma}
\label{lem:existence}
There exists a homomorphism~$H_1(P \times D^2\setminus R(\tau_1,\tau_2))\rightarrow G$ which is trivial when restricted to loops in~$ P\times\lbrace x \rbrace$
(with~$x\in\partial D^2$), and whose composition with the homomorphism induced by the inclusion of any copy of~$X_{\tau_i}$ into~$P \times D^2\setminus R(\tau_1,\tau_2)$
coincides with the natural map~$H_1(X_{\tau_i})\to G$ (for~$i=1,2$).
\end{lemma} 

\begin{proof}
Decompose the space~$X=(P \times D^2)\setminus R(\tau_1,\tau_2)$ as the union of~$A=(P\setminus(J_1 \cup J_2))\times D_c$
and~$B=((J_1 \times D^2)\setminus(I_1 \times \tau_1)) \sqcup ((J_2 \times D^2)\setminus(I_2 \times \tau_2))$. As~$P\setminus(J_1 \cup J_2)$ is contractible,~$A$ retracts
onto~$D_c$. As~$B$ is equal to~$(I_1 \times X_{\tau_1}) \sqcup (I_2 \times X_{\tau_2}) $, it has the homotopy type of~$X_{\tau_1} \sqcup X_{\tau_2}$. Finally,~$A \cap B$ has the homotopy type of four disjoint copies of the punctured disc~$D_c$. Therefore the associated Mayer-Vietoris exact sequence has the form
\[
\textstyle H_1(\bigsqcup_1^4 D_c)\to H_1(D_c)\oplus H_1(X_{\tau_1})\oplus H_1(X_{\tau_2})\to H_1(X)\to H_0(\bigsqcup_1^4 D_c)\,,
\]
which allows us to extend~$H_1(X_{\tau_1})\oplus H_1(X_{\tau_2})\rightarrow G$ to the desired map~$H_1(X) \rightarrow G$.
\end{proof}

Using the homomorphism of Lemma~\ref{lem:existence}, one obtains a~$G$-covering~$P_G(\tau_1,\tau_2)\rightarrow P \times D^2$
branched along~$R(\tau_1,\tau_2)$. Let us start by studying its boundary, which is nothing but the lift
of~$\partial(P\times D^2)=\partial P\times D^2\cup_{\partial P\times\partial D^2} P\times\partial D^2$.
By definition, the surface~$R(\tau_1,\tau_2)$ intersects the three components of~$\partial P\times D^2$ in the closure of the three tangles~$\tau_1$,~$\tau_2$ and~$\tau_1\tau_2$ in
solid tori~$S^1\times D^2$. Therefore, Lemma~\ref{lem:existence} implies that~$\partial P\times D^2$ lifts to
\[
\widehat{X}_{\widehat{\tau_1}} \sqcup \widehat{X}_{\widehat{\tau_2}} \sqcup \widehat{X}_{\widehat{\tau_1 \tau_2}}\subset \partial P_G(\tau_1,\tau_2)\,,
\]
where~$X_{\widehat{\tau}}$ denotes the exterior of~$\widehat{\tau}$ in~$S^1\times D^2$ and~$\widehat{X}_{\widehat{\tau}}$ the corresponding cover.
Since~$\omega$ belongs to~$\mathbb{T}_c^\mu\cap\mathbb{T}_P^\mu$, Lemma~\ref{lemma: boundary} ensures that the boundary of each of these components is a single torus.
For the same reason, together with the first condition in Lemma~\ref{lem:existence},~$P\times\partial D^2$ lifts to a single copy
of~$P\times\partial D^2\subset\partial P_G(\tau_1,\tau_2)$. Combining these remarks, we get
\[
\partial P_G(\tau_1,\tau_2)=(\widehat{X}_{\widehat{\tau_1}}\sqcup\widehat{X}_{\widehat{\tau_2}}\sqcup\widehat{X}_{\widehat{\tau_1 \tau_2}})\cup_{\partial P\times\partial D^2}(P \times\partial D^2)\,.
\]

Before applying the Novikov-Wall theorem, we must slightly modify~$P_G(\tau_1,\tau_2)$, as follows.
Consider the space~$\widetilde{P}_G(\tau_1,\tau_2)$ given by
\[
\widetilde{P}_G(\tau_1,\tau_2)=P_G(\tau_1,\tau_2) \cup_{P \times\partial D^2} (P \times D^2)\,.
\]
By the discussion above, this manifold has boundary
\[
\partial \widetilde{P}_G(\tau_1,\tau_2)=\widehat{Y}_{\widehat{\tau_1}} \sqcup \widehat{Y}_{\widehat{\tau_2}} \sqcup \widehat{Y}_{\widehat{\tau_1 \tau_2}}\,,
\]
where~$\widehat{Y}_{\widehat{\tau}}$ is the {\em closed\/}~$3$-manifold given by~$\widehat{Y}_{\widehat{\tau}}=\widehat{X}_{\widehat{\tau}}\cup_{S^1\times\partial D^2}(S^1\times D^2)$.

\begin{proposition}
\label{prop: P4k}
For any~$\omega$ and~$\tau_1,\tau_2$ as above, we have
\[
\sigma_\omega(P_G(\tau_1,\tau_2)) \simeq \sigma_\omega(\widetilde{P}_G(\tau_1,\tau_2)) =
\mathit{Maslov}(\mathcal{F}_\omega(\overline{\tau}_1),\Delta, \mathcal{F}_\omega(\tau_2))\,.
\]
\end{proposition}

\begin{proof}
Let us start by applying Novikov-Wall to~$X_0=P \times\partial D^2\subset\widetilde{P}_G(\tau_1,\tau_2))=M$. Note that the other corresponding spaces are given
by~$M_1=P_G(\tau_1,\tau_2)$,~$M_2=P\times D^2$ whose signature vanishes as it has no degree~$2$ homology, while~$\Sigma$ consists of a union of three tori.
Therefore, we immediately obtain that the difference between the~$\omega$-signatures of~$P_G(\tau_1,\tau_2)$ and~$\widetilde{P}_G(\tau_1,\tau_2)$ is uniformly bounded.

To show the second equality, start by cutting the pair of pants~$P$ along the path~$\gamma$ illustrated in Figure~\ref{fig:pants}.
This splits~$P$ into two cylinders~$C_1=I_1 \times S^1$ and~$C_2=I_2 \times S^1$; let us analyze the corresponding splitting of the manifold~$\widetilde{P}_G(\tau_1,\tau_2)$.
By construction,~$\gamma\times D^2\subset P\times D^2$ lifts to~$\gamma\times\widehat{D}_c=\widehat{X}_{\mathit{id}_c}$. In~$\widetilde{P}_G(\tau_1,\tau_2)$, the corresponding manifold is~$X_0:=\widehat{Y}_{\mathit{id}_c}$, whose boundary is given by~$\Sigma := \partial X_0$ which consists of two copies of~$\widehat{D}_c\cup_{\partial D^2} D^2$.
Similarly, the space~$C_i\times D^2\subset P\times D^2$ lifts to~$I_i \times \widehat{X}_{\widehat{\tau_i}}\subset P_G(\tau_1,\tau_2)$ (for~$i=1,2$).
In~$\widetilde{P}_G(\tau_1,\tau_2)$, the corresponding manifold is~$M_i:=I_i \times \widehat{Y}_{\widehat{\tau_i}}$.
As these manifolds are of the form~$[0,1] \times N^3$, for some~$3$-manifold~$N^3$, their signature vanishes. Moreover, the manifolds~$\widehat{Y}_{\widehat{\tau_i}}$ being closed,  we can apply the Novikov-Wall additivity theorem. Writing~$X_i=\partial M_i\setminus X_0$ for~$i=1,2$, we obtain
\begin{align*}
\sigma_\omega(\widetilde{P}_G(\tau_1,\tau_2))&= \sigma_\omega(M_1)+\sigma_\omega(M_2)+\mathit{Maslov}((L_1)_{\omega},(L_0)_{\omega},(L_2)_{\omega})\\
	&=\mathit{Maslov}((L_1)_{\omega},(L_0)_{\omega},(L_2)_{\omega})\,,
\end{align*}
where~$(L_i)_{\omega}$ is the kernel of the map induced by the inclusion of~$\Sigma$ in~$X_i$ ($i=0,1,2$). We now determine these spaces~$(L_i)_{\omega}$.

First, one can check that
\[
\partial X_0=\partial X_1 =\partial X_2 = \Sigma = (\widehat{D}_c\cup_{\partial D^2} D^2)\sqcup(\widehat{D}_c\cup_{\partial D^2} D^2)\,.
\]
As~$\widehat{D}_c$ is a compact orientable surface with one boundary component, its first homology is unaffected by capping off its boundary with a disk.
Therefore, the spaces~$H_1(\Sigma)$ and~$H_1(\widehat{D}_c) \oplus H_1(\widehat{D}_c)$ are canonically isomorphic, and so are the corresponding generalized eigenspaces.
An easy Mayer-Vietoris argument shows that the inclusion~$\widehat{X}_{\mathit{id}_c}\subset X_0$ induces an isomorphism on the first homology.
Therefore, the map~$H_1(\Sigma)_{\omega} \rightarrow H_1(X_0)_{\omega}$ can be identified with the inclusion induced map
\[
j_0\colon H_1(\widehat{D}_c)_{\omega} \oplus H_1(\widehat{D}_c)_{\omega} \cong H_1(\partial \widehat{X}_{\mathit{id}_c} )_{\omega} \rightarrow H_1(\widehat{X}_{\mathit{id}_c})_{\omega}\,.
\]
Similarly, the map~$H_1(\Sigma)_{\omega} \rightarrow H_1(X_2)_{\omega}$ can be identified with the inclusion induced map
\[
j_2\colon H_1(\widehat{D}_c)_{\omega} \oplus H_1(\widehat{D}_c)_{\omega} \cong H_1(\partial \widehat{X}_{\tau_2} )_{\omega} \rightarrow H_1(\widehat{X}_{\tau_2})_{\omega}
\]
while~$H_1(\Sigma)_{\omega} \rightarrow H_1(X_1)_{\omega}$ is the inclusion induced map
\[
j_1\colon H_1(\widehat{D}_c)_{\omega} \oplus H_1(\widehat{D}_c)_{\omega} \cong H_1(\partial \widehat{X}_{\overline{\tau}_1} )_{\omega} \rightarrow H_1(\widehat{X}_{\overline{\tau}_1})_{\omega}\,.
\]
(The appearance of the reflection of~$\tau_1$ should be clear from Figure~\ref{fig:pants}.)

Summarizing, we have shown that
\[
\sigma_\omega(\widetilde{P}_G(\tau_1, \tau_2))=\mathit{Maslov}(\mathit{Ker}(j_1),\mathit{Ker}(j_0),\mathit{Ker}(j_2))\,.
\]
Applying Remark~\ref{rem: Fcharac} to each of the three tangles~$\overline{\tau}_1$,~$\mathit{id}_c$ and~$\tau_2$, and the last point of Lemma~\ref{lemma:Maslov}
to the unitary involution~$\psi=(-\mathit{id})\oplus\mathit{id}$, we have
\begin{align*}
\sigma_\omega(\widetilde{P}_G(\tau_1, \tau_2))&=\mathit{Maslov}(\psi(\mathcal{F}_\omega(\overline{\tau}_1)),\psi(\mathcal{F}_\omega(\mathit{id}_c)),\psi(\mathcal{F}_\omega(\tau_2)))\\
	&=\mathit{Maslov}(\mathcal{F}_\omega(\overline{\tau}_1),\mathcal{F}_\omega(\mathit{id}_c),\mathcal{F}_\omega(\tau_2))\,,
\end{align*}
and the proof is completed.
\end{proof}

\subsection{The manifold~$C_G(\tau)$.}
\label{sub:C}

Next, we build the manifold that encodes the signature of the tangle closure.
This will require some notations. Let~$D^4$ denote the (oriented) unit~$4$-ball,~$S^3=\partial D^4$ its oriented boundary, and~$T=S^1\times D^2\subset S^3$ the standardly embedded solid torus.

Closing a colored tangle~$\tau\subset [0,1]\times D^2$ yields a colored link~$\widehat{\tau}\subset T$.
Consider a collection~$S(\tau)$ of surfaces that bound~$\widehat{\tau}\subset S^3$ and that are ``in general position'' in~$D^4$.
In other words,~$S(\tau)$ consists of a collection of surfaces~$S_1\cup\dots\cup S_\mu$ smoothly embedded in~$D^4$, whose only intersections are transverse double points (between
different surfaces), and such that for all~$i$,~$S_i$ meets~$S^3=\partial D^4$ along the sublink of~$\widehat{\tau}\subset T\subset S^3$ of color~$i$.
(Such a surface can be obtained, for example, by taking any~$C$-complex for~$\widehat{\tau}$ and by pushing it inside the~$4$-ball.)

Let us further assume that~$S(\tau)$ meets the radius one-half sphere~$\frac{1}{2}S^3$ along the closure of the trivial tangle~$\mathit{id}_c$ (i.e. the~$n$-component unlink)
in a way that respects the coloring~$c$. Finally, we shall assume that the intersection of~$S(\tau)$ with the closure of~$D^4\setminus\frac{1}{2}D^4$ is contained in
the subspace~$N$ of~$\mathit{cl}(D^4\setminus\frac{1}{2}D^4)\cong [0,1]\times S^3$ given by
\[
N = \lbrace x \in D^4 \ | \ 1/2 \leq ||x|| \leq 1, \ x/||x|| \in T \rbrace\, \cong \, [0,1] \times T\,.
\]
One easily checks that such a surface can be obtained by pushing a~$C$-complex for~$\widehat{\tau}$ inside~$D^4$ and isotopying it in the appropriate way.
 
A standard computation shows that~$H_1(D^4 \setminus S(\tau))$ is free abelian of rank~$\mu$. Let
\[
C_G(\tau)\rightarrow N
\]
be the~$G$-cover branched over~$S(\tau) \cap N$ induced by the composition of the inclusion induced
homomorphism~$H_1(N \setminus(S(\tau) \cap N)) \rightarrow H_1(D^4 \setminus S(\tau))$ with the canonical projection~$H_1(D^4 \setminus S(\tau))\rightarrow G$.
Let us analyse its boundary. Writing~$C=[0,1] \times S^1$, the boundary of~$N=C \times D^2$ can be written
as~$\partial(C\times D^2)=(\partial C \times D^2) \cup_{\partial C \times \partial D^2} (C \times \partial D^2)$.
Thanks to the conditions stated above,~$\partial C \times D^2$ lifts
to~$\widehat{X}_{\widehat{\tau}}\sqcup \widehat{X}_{\widehat{\mathit{id}_c}}$. On the other hand, as~$\omega$ is
in~$\mathbb{T}^\mu_c \cap \mathbb{T}^\mu_P$, the space~$C \times \partial D^2$ lifts to a single copy of~$C \times \partial D^2$. Summarizing, we get
\[
\partial C_G(\tau) = (\widehat{X}_{\widehat{\tau}} \sqcup \widehat{X}_{\widehat{\mathit{id}_c}}) \cup_{\partial C \times \partial D^2} (C \times \partial D^2)\,.
\]
We are now ready to compute the~$\omega$-signature of~$C_G(\tau)$.

\begin{proposition}
\label{prop:C4k}
For any~$\omega$ and~$\tau$ as above,~$\sigma_\omega(C_G(\tau))\simeq\sign_\omega(\widehat{\tau})$.
\end{proposition}

\begin{proof}
Let~$W_{\widehat{\tau}} \rightarrow D^4$ be the~$G$-cover of~$D^4$ branched along~$S(\tau)\subset D^4$ given by the
homomorphism~$H_1(D^4 \setminus S(\tau))\rightarrow G$. By~\cite[Theorem 6.1]{CF},~$\sigma_\omega(W_{\widehat{\tau}})=\sign_\omega(\widehat{\tau})$, so we are left with the proof
that~$\sigma_\omega(W_{\widehat{\tau}})\simeq\sigma_\omega(C_G(\tau))$. To do so, we will apply the Novikov-Wall theorem twice.

First, the space~$\mathit{cl}(D^4 \setminus\frac{1}{2}D^4)\cong [0,1]\times S^3$ can be obtained by gluing a copy of~$[0,1]\times D^2\times \partial D^2$ to~$N$
along~$[0,1]\times S^1\times\partial D^2$.
Lifting this to the covers, one gets a manifold~$M$ obtained by gluing a copy of~$[0,1]\times D^2\times\partial D^2$ to~$C_G(\tau)$ along~$X_0:=[0,1]\times S^1\times\partial D^2$,
with~$\Sigma:=\partial X_0$ consisting of two disjoint tori.
As the boundary of~$C_G(\tau)$ is~$(\widehat{X}_{\widehat{\tau}} \sqcup \widehat{X}_{\widehat{\mathit{id}_c}}) \cup_{\Sigma} X_0$
while the boundary of~$[0,1]\times D^2\times \partial D^2$ is~$(\{0,1\}\times D^2\times \partial D^2)\cup_{\Sigma} X_0$, Novikov-Wall yields 
\[
\sigma_\omega(M) \simeq \sigma_\omega(C_G(\tau))+\sigma_\omega([0,1]\times D^2\times\partial D^2)=\sigma_\omega(C_G(\tau))\,.
\]
Next, glue the ball~$\frac{1}{2}D^4$ to~$\mathit{cl}(D^4 \setminus \frac{1}{2}D^4)$ along~$\frac{1}{2}S^3$ in order to obtain~$D^4$. Lifting this to the covers, it corresponds to
recovering the manifold~$W_{\widehat{\tau}}$ by gluing~$W_{\widehat{\mathit{id}_c}}$ to~$M$ along the preimage of~$\frac{1}{2}S^3$. As the latter space is closed,
Novikov-Wall additivity applies trivially and we get 
\[
\sigma_\omega(W_{\widehat{\tau}})=\sigma_\omega(M)+\sigma_\omega(W_{\widehat{\mathit{id}_c}})=\sigma_\omega(M)+\sign_\omega(\widehat{\mathit{id}_c})=\sigma_\omega(M)\,.
\]
This concludes the proof.
\end{proof}

\subsection{The manifold~$M_G(\tau_1,\tau_2)$.}
\label{sub:M}

\begin{figure}[tb]
\labellist\small\hair 2.5pt
\pinlabel {$P_G(\tau_1,\tau_2)$} at 172 278
\pinlabel {$P_G(\mathit{id},\mathit{id})$} at 157 51
\pinlabel {$C_G(\tau_1)$} at 56 204
\pinlabel {$C_G(\tau_2)$} at 625 203
\pinlabel {$C_G(\tau_1\tau_2)$} at 320 185
\pinlabel {$\gamma$} at 348 24
\endlabellist
\centering
\includegraphics[width=0.7\textwidth]{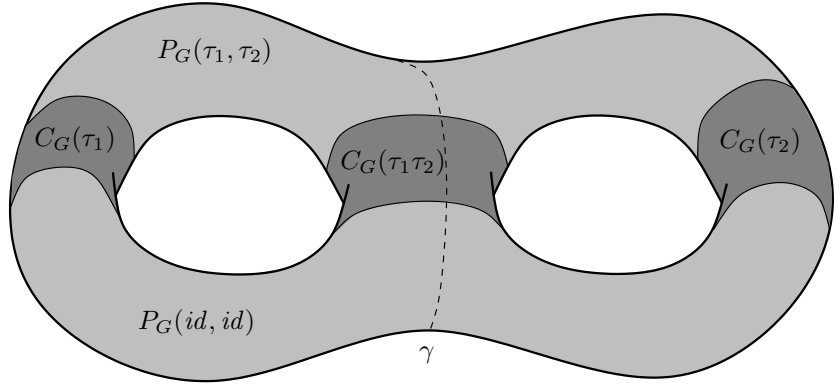}
\caption{The manifold~$M_G(\tau_1,\tau_2)$.}
\label{fig:M}
\end{figure}

Our goal is now to glue several copies of the manifolds~$C_G(\tau)$ and~$P_G(\tau_1, \tau_2)$ along
their boundary in order to obtain an oriented~$4$-manifold~$M_G(\tau_1,\tau_2)$. Recall that these boundaries are given by
\[
\partial C_G(\tau)=(\widehat{X}_{\widehat{\tau}} \sqcup \widehat{X}_{\widehat{\mathit{id}_c}}) \cup_{\partial C\times\partial D^2} (C\times\partial D^2)\,,
\]
with~$C=[0,1] \times S^1$, while
\[
\partial P_G(\tau_1,\tau_2)=(\widehat{X}_{\widehat{\tau_1}}\sqcup\widehat{X}_{\widehat{\tau_2}}\sqcup\widehat{X}_{\widehat{\tau_1\tau_2}})\cup_{\partial P \times \partial D^2} (P \times \partial D^2)\,,
\]
where~$P$ denotes the pair of pants. It therefore makes sense to define the manifold~$M_G(\tau_1, \tau_2)$ by gluing~$P_G(\tau_1, \tau_2) $ ``on one side'' of the disjoint union of~$C_G(\tau_1)$,~$C_G(\tau_2)$ and~$ C_G(\tau_1 \tau_2)$, and~$P_G(\mathit{id}_c,\mathit{id}_c)$ ``on the other side'' (see Figure~\ref{fig:M}).
More precisely, set
\[
M_G(\tau_1, \tau_2)=P_G(\tau_1,\tau_2) \cup_{\widehat{X}_{\widehat{\tau_1}} \sqcup \widehat{X}_{\widehat{\tau_2}} \sqcup \widehat{X}_{\widehat{\tau_1 \tau_2}}} (C_G(\tau_1) \sqcup C_G(\tau_2) \sqcup C_G(\tau_1 \tau_2)) \cup_{\widehat{X}_{\widehat{\mathit{id}_c}} \sqcup \widehat{X}_{\widehat{\mathit{id}_c}} \sqcup \widehat{X}_{\widehat{\mathit{id}_c}}} P_G(\mathit{id}_c,\mathit{id}_c)\,.
\]
By construction, this~$4$-manifold is a covering of
\[
(P \times D^2) \cup_{\partial P \times D^2} (C\times D^2 \sqcup C\times D^2 \sqcup C\times D^2 ) \cup_{\partial P \times D^2} (P \times D^2)=\Sigma_2\times D^2\,,
\]
where~$\Sigma_2$ is the closed orientable surface of genus~$2$ (see Figure~\ref{fig:M}), branched over
\[
T(\tau_1, \tau_2 )=R(\tau_1, \tau_2) \cup_{\widehat{\tau_1} \sqcup \widehat{\tau_2} \sqcup \widehat{\tau_1 \tau_2} } (S(\tau_1) \sqcup  S(\tau_2) \sqcup S(\tau_1 \tau_2)) \cup_{\widehat{\mathit{id}_c} \sqcup \widehat{\mathit{id}_c} \sqcup \widehat{\mathit{id}_c} } R(\mathit{id}_c,\mathit{id}_c)\,.
\]
Moreover, its~$\omega$-signature is precisely what we wish to bound, as shown by the following proposition.
 
\begin{proposition}
\label{prop: signatureM}
The~$4$-manifold~$M_G(\tau_1,\tau_2)$ can be endowed with an orientation, so that
\[
\sigma_\omega(M_G(\tau_1,\tau_2))\simeq\sign(\widehat{\tau_1 \tau_2})-sign_\omega(\widehat{\tau_1})-sign_\omega(\widehat{\tau_2})-\mathit{Maslov}(\mathcal{F}_\omega(\overline{\tau}_1),\mathcal{F}_\omega(\mathit{id}_c),\mathcal{F}_\omega(\tau_2))\,.
\]
\end{proposition}

\begin{proof}
We need to be more precise about the orientation of the~$4$-manifolds in play. First note that any (arbitrary but fixed) orientation on the cylinder~$C=[0,1]\times S^1$ and on
the unit disk~$D^2$ defines an orientation on their product~$C\times D^2$, which lifts to an orientation on the cover~$C_G(\tau)$, and induces an orientation
on~$\widehat{X}_{\widehat{\tau}}\sqcup\widehat{X}_{\widehat{\mathit{id}_c}}\subset\partial C_G(\tau)$. However, note that these two spaces are now endowed with {\em opposite\/}
orientations (with respect to a fixed orientation of the solid torus~$S^1\times D^2$ which lifts to an orientation of~$\widehat{X}_{\widehat{\tau}}$ for any tangle~$\tau$).
This can be written
\[
\partial C_G(\tau)\supset\widehat{X}_{\widehat{\tau}} \sqcup -\widehat{X}_{\widehat{\mathit{id}_c}}\,.
\]
By the same arguments, the fixed orientation on the pair of pants~$P$ (and on~$D^2$) induces an orientation on~$P_G(\tau_1,\tau_2)$ such that
\[
\partial P_G(\tau_1,\tau_2)\supset\widehat{X}_{\widehat{\tau_1}}\sqcup\widehat{X}_{\widehat{\tau_2}}\sqcup-\widehat{X}_{\widehat{\tau_1\tau_2}}\,.
\]
Therefore, for the manifold~$M_G(\tau_1,\tau_2)$ to be oriented, we need to paste {\em positively oriented\/} copies of~$P_G(\tau_1,\tau_2)$ and~$C_G(\tau_1\tau_2)$ together with
{\em negatively oriented\/} copies of~$C_G(\tau_1)$,~$C_G(\tau_2)$ and~$P_G(\mathit{id}_c,\mathit{id}_c)$ (or the opposite). It only remains to apply the Novikov-Wall
theorem a couple of times, as follows.
 
Let~$M$ be the manifold obtained by gluing~$P_G(\tau_1, \tau_2)$ to~$-C_G(\tau_1) \sqcup -C_G(\tau_2)\sqcup C_G(\tau_1 \tau_2)$ along
the~$3$-manifold~$X_0:=\widehat{X}_{\widehat{\tau_1}} \sqcup \widehat{X}_{\widehat{\tau_2}} \sqcup-\widehat{X}_{\widehat{\tau_1 \tau_2}}$ whose boundary~$\Sigma$ consists of~$3$ disjoint tori. By Novikov-Wall, Proposition~\ref{prop: P4k} and Proposition~\ref{prop:C4k}, we have
\begin{align*}
\sigma_\omega(M) &\simeq \sigma_\omega(P_G(\tau_1,\tau_2)) + \sigma_\omega(C_G(\tau_1 \tau_2))-\sigma_\omega(C_G(\tau_1))-\sigma_\omega(C_G(\tau_2))\\
&\simeq\mathit{Maslov}(\mathcal{F}_\omega(\overline{\tau}_1),\Delta,\mathcal{F}_\omega(\tau_2))+\sign(\widehat{\tau_1 \tau_2})-\sign_\omega(\widehat{\tau_1})-\sign_\omega(\widehat{\tau_2}).
\end{align*}
Applying the exact same line of reasoning to the gluing of~$P_G(\mathit{id}_c, \mathit{id}_c)$, the result follows from Proposition~\ref{prop: P4k} as
\[
\sigma_\omega(P_G(\mathit{id}_c, \mathit{id}_c))\simeq\mathit{Maslov}(\Delta,\Delta,\Delta)=0\,,
\]
by the second point of Lemma~\ref{lemma:Maslov}.
\end{proof}

To prove Theorem~\ref{thm:main}, it only remains to show that the~$\omega$-signature of~$M_G(\tau_1,\tau_2)$ vanishes up to an uniformly bounded additive constant.
This requires a small lemma.

\begin{lemma}
\label{lem: curve}
For well-chosen surfaces~$S(\tau_1)$,~$S(\tau_2)$ and~$S(\tau_1\tau_2)$ in the construction above,
the branched covering~$M_G(\tau_1,\tau_2)\to \Sigma_2\times D^2$ satisfies the following property:
there exists a curve~$\gamma$ in the genus~$2$ surface~$\Sigma_2$ such~$\gamma \times D^2$ intersects the branch set~$T(\tau_1, \tau_2)\subset\Sigma_2\times D^2$ in the~$n$ disjoint circles~$\gamma\times\lbrace x_1,\dots,x_n\rbrace$.
\end{lemma}

\begin{proof}
Fix~$C$-complexes~$S(\tau_1)$ and~$S(\tau_2)$ for the links~$\widehat{\tau_1}$ and~$\widehat{\tau_2}$, and build a~$C$-complex~$S(\tau_1 \tau_2)$ for~$\widehat{\tau_1 \tau_2}$
by connecting~$S(\tau_1)$ and~$S(\tau_2)$ along disjoint bands far from the tangles. Let us use the same notation for the surfaces obtained by pushing these~$C$-complex inside~$D^4$ 
in such a way that they intersect~$[0,1]\times\lbrace s_1\rbrace\times D^2$ along~$[0,1]\times\lbrace s_1\rbrace\times\lbrace x_1,\dots,x_n\rbrace$, where~$s_1$ is some point on the circle~$S^1$ far from the tangle.

The strategy is to build the curve~$\gamma$ from four intervals~$\gamma_1,\gamma_2,\gamma_3,\gamma_4$ alternatively contained in the pairs of pants and in the central cylinder,
as illustrated in Figure~\ref{fig:M}. For the branch set~$R(\tau_1,\tau_2)$ (resp.~$R(\mathit{id}_c,\mathit{id}_c)$), one can simply pick an interval~$\gamma_1$ (resp.~$\gamma_3$)
in the pair of pants as illustrated in Figure~\ref{fig:pants}.
For the branch set~$S(\tau_1 \tau_2)$, one must find two intervals satisfying the same property as above. Using the way we pushed the~$C$-complex~$S(\tau_1 \tau_2)$ into the~$4$-ball, we can set~$\gamma_2=[0,1] \times \lbrace s_1 \rbrace$. Finally, using the way we built~$S(\tau_1 \tau_2)$ from~$S(\tau_1)$ and~$S(\tau_2)$, there exists a
second point~$s_2 \neq s_1$ such that the~$C$-complex intersects~$[0,1]\times\lbrace s_2 \rbrace \times D^2$ along~$[0,1]\times\lbrace s_2 \rbrace \times \lbrace x_1,\dots,x_n\rbrace$. Set~$\gamma_3=[0,1] \times \lbrace s_2 \rbrace$. Gluing these intervals~$\gamma_1,\dots,\gamma_4$ together produces the required curve~$\gamma$.
\end{proof}

We may now conclude.

\begin{proposition}
\label{prop:vanish}
The~$\omega$-signature of~$M_G(\tau_1,\tau_2)$ vanishes up to an uniformly bounded constant.
\end{proposition}

\begin{proof}
Cutting the genus~$2$ surface~$\Sigma_2$ along the curve~$\gamma$ provided by Lemma~\ref{lem: curve} yields a decomposition~$\Sigma_2=\Sigma_1 \cup_\gamma \Sigma_1$,
where~$\Sigma_1$ denotes the genus~$1$ surface with one boundary component. The induced decomposition of~$\Sigma_2 \times D^2$ lifts to 
\[
M_G(\tau_1,\tau_2)=Q_G(\tau_1)\cup_{X_0} Q_G(\tau_2)\,,
\]
with~$X_0=\gamma \times \widehat{D}_c$ thanks to the way we chose~$\gamma$. Applying the Novikov-Wall theorem to this decomposition, we get
\[
\sigma_\omega(M_G(\tau_1,\tau_2))\simeq\sigma_\omega(Q_G(\tau_1))+\sigma_\omega (Q_G(\tau_2))\,.
\]
It remains to show that~$\sigma_\omega(Q_G(\tau))\simeq 0$ for any tangle~$\tau$. Using the second point of Lemma~\ref{lemma:Maslov} together with Proposition~\ref{prop: signatureM}
and the equation displayed above for~$(\tau_1,\tau_2)=(\tau,\mathit{id}_c)$, we have 
\begin{align*}
0&=\sigma_\omega(\widehat{\tau})-\sigma_\omega(\widehat{\tau})-\sigma_\omega(\widehat{\mathit{id}_c})-\mathit{Maslov}(\mathcal{F}_\omega(\overline{\tau}),\Delta,\Delta)\\
	&\simeq \sigma_\omega(M_G(\tau,\mathit{id}_c))\simeq \sigma_\omega(Q_G(\tau))+\sigma_\omega(Q_G(\mathit{id}_c))
\end{align*}
independently of the tangle~$\tau$. Taking~$\tau=\mathit{id}_c$ yields the result.
\end{proof}


\section{The isotropic functor as an evaluation}
\label{sec:rel}

In this section, we outline the construction of the Lagrangian functor defined in~\cite{CT} and relate it to our isotropic functor.
In particular, we will see in Theorem~\ref{thm:relating} below why the isotropic functor is, in some sense, an evaluation of the Lagrangian functor.
Recall that these results were used in Section~\ref{sec:def} to compute several examples, and to prove Corollary~\ref{cor:braid}.

\subsection{The Lagrangian functor}
\label{sub:Lagr}
In all this subsection, we shall assume for simplicity that the maps~$c\colon\lbrace 1,\dots,n\rbrace \rightarrow \lbrace \pm 1,\dots,\pm \mu \rbrace$ are such that~$\ell(c)$
is nowhere vanishing.
Given such a map~$c$, the homomorphism~$H_1(D_c) \rightarrow C_\infty^\mu$,~$e_j \mapsto t_{|c_j|}$
induces a free abelian covering~$\widetilde{D}_c \rightarrow D_c$ whose homology is endowed with a structure of module over~$\Lambda_\mu=\Z[t_1^{\pm 1},\dots,t_\mu^{\pm 1}]$.
Obviously,~$D_c$ retracts by deformation on the wedge of~$n$ circles representing the generators~$e_1,\dots,e_n$ of~$\pi_1(D_c,z)$. Using this fact,
one can easily check that~$H_1(\widetilde{D}_c)$ is a module of rank~$n-1$ which is free over~$\Lambda_\mu$ if~$\mu<3$, and
whose localization with respect to the multiplicative set~$S\subset\Lambda_\mu$ generated by~$t_1-1,\dots,t_\mu-1$ is always free over the localized
ring~$\Lambda_S=\Z[t_1^{\pm 1},\dots,t_\mu^{\pm 1},(t_1-1)^{-1},\dots,(t_\mu-1)^{-1}]$.
 
Let~$\langle \ , \ \rangle \colon H_1(\widetilde{D}_c)\times H_1(\widetilde{D}_c)\rightarrow\Z$ be the skew-symmetric intersection form obtained by lifting the orientation of~$D_c$
to~$\widetilde{D}_c$. As showed in~\cite{CT}, the formula
\[
\xi_c(x,y)=\sum_{g\in C_\infty^\mu}\langle gx,y \rangle g^{-1}
\]
defines a skew-Hermitian~$\Lambda_\mu$-valued pairing on~$H_1(\widetilde{D}_c)$ which is non-degenerate. Therefore, following the terminology of
subsection~\ref{sub:cat},~$(H_1(\widetilde{D}_c),\xi_c)$ is a non-degenerate Hermitian~$\Lambda_\mu$-module (which is free if~$\mu<3$).

As promised in subsection~\ref{sub:functor}, we now compute a couple of examples.

\begin{example}
\label{ex:form12}
In the case~$\mu=1$, if~$\tilde{e}_1,\dots ,\tilde{e}_n$ are the lifts of the loops~$e_1,\dots ,e_n$ starting at some fixed lift~$\tilde{z}$ of the basepoint~$z$,
then a basis of the free~$\Lambda$-module~$H_1(\widetilde{D}_c)$ is given by~$v_i=\tilde{e}_i-\tilde{e}_{i+1}$ for~$i=1, \dots, n-1$. With respect to this basis,
a matrix for the skew-Hermitian intersection form~$\xi_c$ is
\[
\begin{pmatrix}
  \frac{1}{2}(\varepsilon_1+\varepsilon_2)(t-t^{-1})& 1-t^{\varepsilon_2} & 0 & \hdots & 0\\
  t^{-\varepsilon_2}-1 & \frac{1}{2}(\varepsilon_2+\varepsilon_3)(t-t^{-1}) &  & \ddots & \vdots \\
  0 &  & \ddots  &  & 0 \\
  \vdots & \ddots &  & & 1-t^{\varepsilon_n}  \\
  0 & \hdots & 0 &t^{-\varepsilon_n}-1 & \frac{1}{2}(\varepsilon_{n-1}+\varepsilon_n)(t-t^{-1})
 \end{pmatrix}.
\]
An illustration of this computation in the case~$c=(1,1)$ is shown in Figure~\ref{fig:coverings}.
\end{example}

\begin{figure}[tb]
\labellist\small\hair 2.5pt
\pinlabel {$v_1$} at 120 820
\pinlabel {$tv_1$} at 387 820
\pinlabel {$v_1$} at 97 -52
\pinlabel {$t_1v_1$} at 307 -52
\endlabellist
\includegraphics[width=0.6\textwidth]{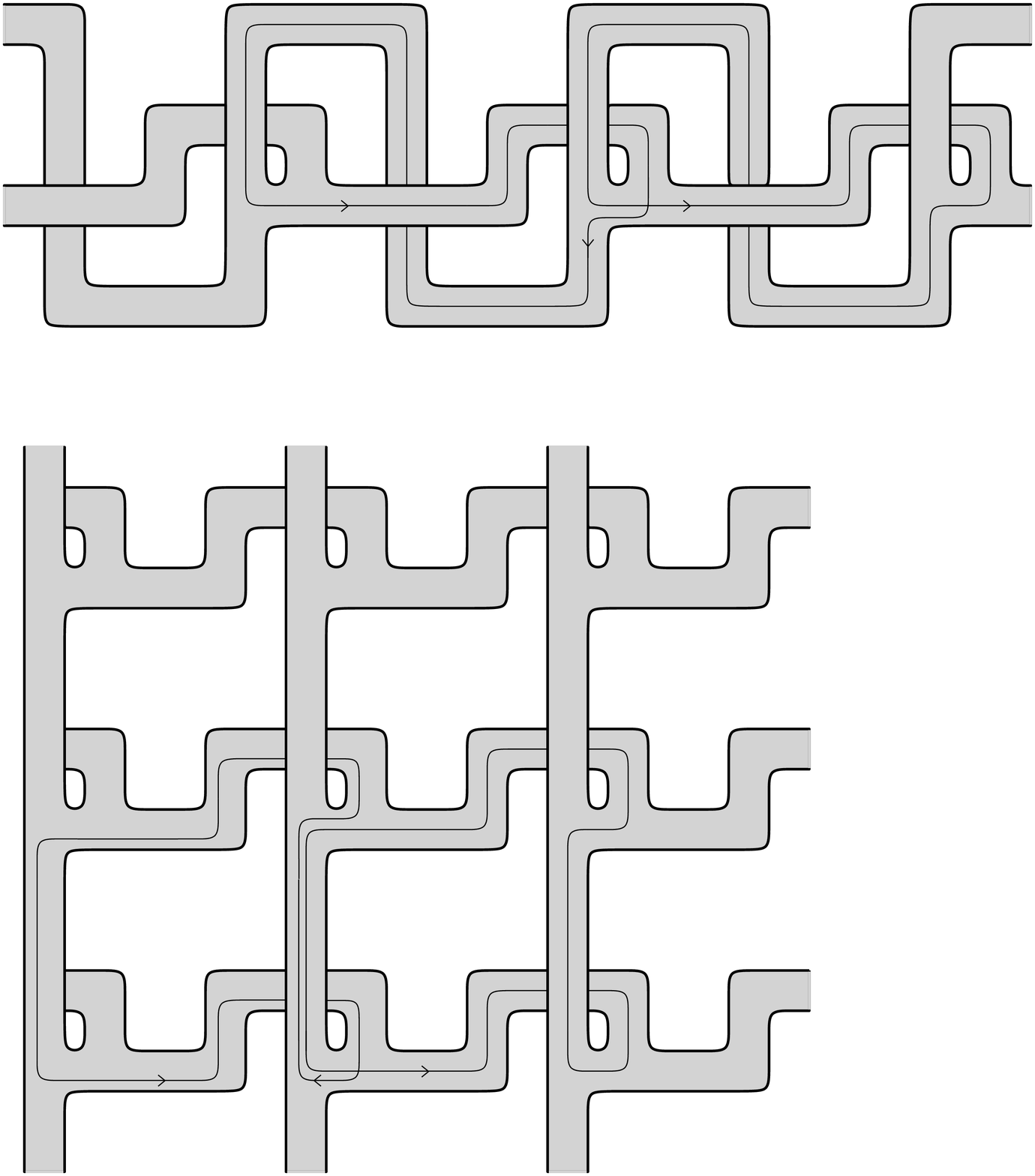}
\caption{From top to bottom: the covering~$\widetilde{D}_c$ for~$c=(1,1)$ and $c=(1,2)$.}
\label{fig:coverings}
\end{figure}

\begin{example}
\label{ex:form2}
Consider the case~$n=\mu=2$ and~$c=(1,2)$. If~$\tilde{e}_1$ and~$\tilde{e}_2$ are lifts of the loops~$e_1$ and~$e_2$ starting at~$\tilde{z}$, then a basis of the
free~$\Lambda_2$-module~$H_1(\widetilde{D}_c)$ is given by~$v=(1-t_2)\tilde{e}_1-(1-t_1)\tilde{e}_2$. With respect to this basis, a matrix for the skew-Hermitian intersection
form~$\xi_c$ is
\[
((t_1-t_1^{-1})+(t_2-t_2^{-1})-(t_1t_2-t_1^{-1}t_2^{-1}))\,.
\]
An illustration of this computation is shown in Figure~\ref{fig:coverings}.
\end{example}

The homomorphism~$H_1(X_\tau)\rightarrow C_\infty^\mu $,~$ m_j \mapsto t_{|c_j|}$ extends the previously defined homomorphisms~$H_1(D_c) \rightarrow C_\infty^\mu$
and~$H_1(D_{c'})\rightarrow C_\infty^\mu$. It determines a free abelian covering~$\widetilde{X}_\tau \rightarrow X_\tau$ whose homology is also endowed with a
structure of module over~$\Lambda_\mu$.

Let~$i_\tau\colon H_1(\widetilde{D}_c) \rightarrow H_1(\widetilde{X}_\tau)$ and~$i_\tau'\colon H_1(\widetilde{D}_{c'}) \rightarrow H_1(\widetilde{X}_\tau)$ be the homomorphisms induced by the inclusions of~$\widetilde{D}_c$ and~$\widetilde{D}_{c'}$ into~$\widetilde{X}_\tau$. Denote
by~$j_\tau\colon H_1(\widetilde{D}_c)\oplus H_1(\widetilde{D}_{c'}) \rightarrow H_1(\widetilde{X}_\tau)$ the homomorphism given by~$j_\tau(x,x')=i_\tau'(x')-i_\tau(x)$. 
Finally consider
\[
\mathit{Ker}(j_\tau) \subset H_1(\widetilde{D}_c) \oplus H_1(\widetilde{D}_{c'})\,.
\]
It is proved in~\cite{CT} that for any tangle~$\tau$, the module~$\overline{\mathit{Ker}(j_\tau)}$ is Lagrangian. It can also be checked
that~$\mathit{Ker}(j_{\tau_1\tau_2})=\mathit{Ker}(j_{\tau_1})\mathit{Ker}(j_{\tau_2})$ for any tangles~$\tau_1, \tau_2$.
This leads to the main result of~\cite{CT}.

\begin{theorem}
\label{thm: CT}
Let~$\mathcal{F}$ assign to each coloring map~$c\colon\lbrace 1,\dots,n \rbrace \rightarrow \lbrace\pm 1,\dots,\pm \mu \rbrace$ the
pair~$(H_1(\widetilde{D}_c),\xi_c)$ and to
each~$(c,c')$-tangle~$\tau$ the submodule~$\overline{\mathit{Ker}(j_\tau)}$ of~$H_1(\widetilde{D}_c)\oplus H_1(\widetilde{D}_{c'})$. Then~$\mathcal{F} $ is a
functor which fits in the commutative diagram
\[
\begin{xy}
(0,15)*+{\mathbf{Braids}_\mu}="a";
(25,15)*+{\mathbf{Tangles}_\mu}="c";
{\ar@{->}^-{} "a";"c"};
(0,0)*+{\mathbf{U}_{\Lambda_\mu}}="f";
(25,0)*+{\mathbf{Lagr}_{\Lambda_\mu},}="h";
{\ar@{->}^-{\Gamma} "f";"h"};
{\ar@{->}^{} "a";"f"};
{\ar@{->}^{\mathcal{F}} "c";"h"};
\end{xy} 
\]
where the horizontal arrows are the embeddings of categories described in Sections~\ref{sub:tangle} and~\ref{sub:cat}.
\end{theorem}

If~$\alpha\in B_c$ is a colored braid, then~$\mathcal{F}(\alpha)$ is precisely the graph of the unitary automorphism of~$H_1(\widetilde{D}_c)$ given by the colored
Gassner automorphism~$\mathcal{B}_t(\alpha)$. 

\begin{example}
\label{ex:coloredG}
Let us consider the case~$\mu=2$ and~$c=(1,2)$. As we saw in Example~\ref{ex:form2}, a basis of the free~$\Lambda_2$-module~$H_1(\widetilde{D}_c)$ is given
by~$v=(1-t_2)\tilde{e}_1-(1-t_1)\tilde{e}_2$. With respect to this basis, the matrix for the colored Gassner representation is easily seen to
be given by~$\mathcal{B}_{(t_1,t_2)}(A_{12})=(t_1t_2)$.
\end{example}

\subsection{Statement of the results}
\label{sub:rel}

Under some mild assumptions on the colored tangle, we shall show how the Lagrangian functor is related to the isotropic functor. In particular, it will follow that if~$\alpha$ is a braid, then~$\mathcal{F}_\omega(\alpha)$ can be understood by evaluating a matrix of the colored Gassner representation at~$t=\omega$. The proofs of these slightly technical
statements will be given in the next paragraph.

Fix an element~$\omega$ in~$\mathbb{T}^\mu$ and assume that the component~$\omega_i$ of~$\omega$ is of finite order~$k_i>1$. Let~$G$ be the finite abelian
group~$C_{k_1} \times\dots\times C_{k_\mu}$. For~$i=1,\dots,\mu$ let~$t_i$ be a generator of~$C_{k_i}$ and let~$\chi_\omega$ be the character of~$G$
sending the generator~$t_i$ to the root of unity~$\omega_i\in\C\setminus\{1\}$. As usual, we shall simply denote by~$H_\omega$ the associated generalized eigenspaces
(recall subsection~\ref{sub:eval}). Note that~$\chi_\omega$ induces a ring homomorphism~$\Lambda_\mu\to\C$ which endows~$\C$ with the structure of a module
over~$\Lambda_\mu$, that will be emphasized by the notation~$\C_{\omega}$. Note also that since~$\omega_i\neq 1$, this homomorphism factors through the localized
ring~$\Lambda_S=\Z[t_1^{\pm 1},\dots,t_\mu^{\pm 1},(t_1-1)^{-1},\dots,(t_\mu-1)^{-1}]$. As customary, we shall write~$H_S$ for the localization~$H\otimes_{\Lambda_\mu}\Lambda_S$
of a~$\Lambda_\mu$-module~$H$. Note the identity~$H\otimes_{\Lambda_\mu}\C_\omega=H_S\otimes_{\Lambda_S}\C_\omega$.

\begin{proposition}
\label{prop:form}
For any~$c\colon\lbrace 1,\dots,n \rbrace \rightarrow \lbrace \pm 1,\dots,\pm \mu \rbrace$, there is a natural isomorphism of complex vector spaces
\[
\psi_c\colon H_1(\widetilde{D}_c)\otimes_{\Lambda_\mu}\C_{\omega}\longrightarrow H_1(\widehat{D}_c)_{\omega}\,.
\]
Furthermore, if~$\xi_c(t)$ is the matrix for the form of~$\mathcal{F}(c)_S$ with respect to some~$\Lambda_S$-basis~$\{v_j\}_j$
of~$H_1(\widetilde{D}_c)_S$, then the matrix for the form~$\langle \ , \ \rangle_c$ of~$\mathcal{F}_\omega(c)$ with respect to the basis~$\{\psi_c(v_j\otimes 1)\}_j$ of~$H_1(\widehat{D}_c)_{\omega}$ is given by the componentwise evaluation of~$\xi_c(t)$ at~$t=\omega$, up to a positive multiplicative constant. 
\end{proposition}

Next, we introduce some terminology. A free submodule~$N$ of~$H \oplus H'$ is determined by a matrix of the inclusion~$N \subset H \oplus H'$ with respect to a basis of~$N$.
Following~\cite{CT2}, we will say that~$N \subset H \oplus H'$ is {\em encoded} by this matrix. For instance, the graph of a linear map~$\gamma\colon H\to H'$ is encoded
by the matrix~$(I \ \  M_\varphi)^T$, where~$M_\varphi$ is a matrix for~$\varphi$ and~$I$ the identity matrix.

\begin{theorem}
\label{thm:relating}
Assume that~$\omega$ is in~$\mathbb{T}^\mu_P$ and let~$\tau\in T_\mu^\omega(c,c')$ be such that
the~$\Lambda_\mu$-module~$\mathit{Ker}(j_\tau)$ is Lagrangian, and its localization is a free~$\Lambda_S$-module.
Then, with respect to the isomorphisms~$\psi_c$ and~$\psi_{c'}$ of Proposition~\ref{prop:form}, an encoding matrix for the complex vector
space~$\mathcal{F}_\omega(\tau)=\mathit{Ker}(\hat{\jmath}_\tau)$ can be obtained by evaluating an encoding matrix for~$\mathcal{F}(\tau)_S=\mathit{Ker}(j_\tau)_S$ at~$t=\omega$.
\end{theorem}

Recall that a tangle is {\em topologically trivial\/} if its exterior is homeomorphic to the exterior of a trivial braid.
It is easy to check that in such a case, the condition in Theorem~\ref{thm:relating} above is always satisfied (see~\cite[Section 4]{CT2}).
In particular, braids are topologically trivial, so we have the following corollary.

\begin{corollary}
For all~$\alpha \in B_c$ with~$\ell(c)$ nowhere zero and~$\omega\in\mathbb{T}^\mu_c \cap \mathbb{T}^\mu_P$,~$\mathcal{F}_\omega(\alpha)=\Gamma_{\mathcal{B}_\omega(\alpha)}$.
\end{corollary}

\subsection{Proofs of Proposition \ref{prop:form} and Theorem \ref{thm:relating}}
\label{sub:proofs}

In order to prove these results, the first step consists in understanding the relation between free abelian coverings and finite branched coverings.
This is the subject of the following two lemmata. For punctured disks~$D_c$ and tangle exteriors~$X_\tau$, we shall denote by~$\overline{D}_c$ and~$\overline{X}_\tau$
the respective unbranched finite abelian coverings. 

\begin{lemma}
\label{lem: branching}
For any~$c$ and~$\tau$ as above, we have natural isomorphisms
\[H_1(\widehat{D}_c)_{\omega} \cong H_1(\overline{D}_c)_{\omega}\quad\text{and}\quad H_1(\widehat{X}_\tau)_{\omega} \cong H_1(\overline{X}_\tau)_{\omega}\,.
\]
\end{lemma}

\begin{proof}
Recall that the first homology group of the punctured disk~$D_c$ is freely generated by the loops~$e_1,\dots,e_n$, where~$e_j$ is a simple loop turning once around the puncture~$x_j$. Let~$\tilde{e}_j$ be a lift of the loop~$e_j$ to~$\overline{D}_c$ for~$j=1,\dots ,n$. By definition, the branched covering~$\widehat{D}_c\rightarrow D^2$ is obtained from the unbranched covering~$\overline{D}_c \rightarrow D_c$ by gluing~$n$ disks to~$D_c$ (in order to recover~$D^2$) and lifting these gluings to the covering.
Applying the Mayer-Vietoris exact sequence to this decomposition of~$\widehat{D}_c$ shows that the inclusion induced homomorphism defines an isomorphism
$ H_1(\overline{D}_c)/V\cong H_1(\widehat{D}_c)$, where~$V$ is the~$\Z[G]$-submodule generated by the loops~$(1+t_i+ \dots +t_i^{k_i-1})\tilde{e}_j$
for~$j=1,\dots,n$ and~$i$ stands for~$c_j$. Since~$\omega_i\neq 1$ is a~$k_i^{\mathit{th}}$-root of unity,~$\chi(1+t_i+ \dots +t_i^{k_i-1})$ vanishes; this implies that~$c_\chi V=0$,
and the conclusion follows from the second point of Proposition~\ref{prop: nondeg}. The case of the tangle exterior can be treated in the same way.
\end{proof}

\begin{lemma}
\label{lemma: coverings}
There are natural isomorphisms
\[
H_1(\widehat{D}_c)_{\omega} \cong  H_1(\widetilde{D}_c) \otimes_{\Lambda_\mu} \C_{\omega}\quad\text{and}\quad
H_1(\widehat{X}_\tau)_{\omega} \cong  H_1(\widetilde{X}_\tau) \otimes_{\Lambda_\mu} \C_{\omega}\,.
\]
\end{lemma}

\begin{proof}
Both statements will be proved by using standard cut and paste arguments. Let~$I_1, \dots, I_n$ be disjoint intervals in the disk~$D^2$ such that for~$j=1,\dots,n$
the interval~$I_j$ joins the~$j^\mathit{th}$ puncture~$x_j$ to the boundary~$\partial D^2$, and let~$N_j=I_j\times [-1,1]$ be a bicollar neighborhood of~$I_j$ in~$D^2$.
Set~$N=\bigcup_{j=1}^n (N_j\cap D_c)$,~$Y=D_c\setminus\bigcup_{j=1}^n\mathit{Int}(N_j)$,~$R=N\cap Y$ and let~$\tilde{p}\colon\widetilde{D}_c\to D_c$ be the free abelian covering map. The decomposition~$D_c=N \cup Y$ leads to the Mayer-Vietoris exact sequence of~$\Lambda_\mu$-modules
\[
H_1(\widetilde{R})\to H_1(\widetilde{N})\oplus H_1(\widetilde{Y})\to H_1(\widetilde{D}_c)\to H_0(\widetilde{R})\to H_0(\widetilde{N})\oplus H_0(\widetilde{Y})\,,
\]
where~$\widetilde{R}$,~$\widetilde{N}$ and~$\widetilde{Y}$ stand for~$\tilde{p}^{-1}(R)$,~$\tilde{p}^{-1}(N)$ and~$\tilde{p}^{-1}(Y)$, respectively.
Writing~$\overline{p}\colon \overline{D}_c\to D_c$ for the finite abelian covering map and repeating the same argument yields the Mayer-Vietoris exact sequence of~$\Z[G]$-modules
\[
H_1(\overline{R})\to H_1(\overline{N})\oplus H_1(\overline{Y})\to H_1(\overline{D}_c)\to H_0(\overline{R})\to H_0(\overline{N})\oplus H_0(\overline{Y})\,,
\]
where~$\overline{R}$,~$\overline{N}$ and~$\overline{Y}$ stand for~$\overline{p}^{-1}(R)$,~$\overline{p}^{-1}(N)$ and~$\overline{p}^{-1}(Y)$, respectively. 
In the free abelian case, the map~$H_0(\widetilde{R})\to H_0(\widetilde{N}) \oplus H_0(\widetilde{Y})$ is injective while in the finite abelian case, the kernel~$V$ of the  corresponding map~$H_0(\overline{R})\to H_0(\overline{N})\oplus H_0(\overline{Y})$ is freely generated by the~$n$ loops~$\{(1+t_i+\dots+t_i^{k_i-1})\tilde{e}_j\}_{j=1}^n$,
where~$i$ stands for~$c_j$. It follows that the first homology groups of these two coverings are related by
\[
H_1(\overline{D}_c) \cong \left( H_1(\widetilde{D}_c) \otimes_{\Lambda_\mu} \Z[G] \right) \oplus V\,.
\]
Since~$\omega_i$ is a~$k_i^\mathit{th}$ root of unity different from~$1$, we have~$V_\omega\cong V\otimes_{\Z[G]}\C_\omega=0$.
Therefore, using Lemma~\ref{lem: branching}, Proposition~\ref{prop: eval eingenspace} and the isomorphism displayed above, one obtains
\[
H_1(\widehat{D}_c)_\omega\cong H_1(\overline{D}_c)_\omega\cong H_1(\overline{D}_c)\otimes_{\Z[G]}\C_\omega\cong H_1(\widetilde{D}_c)\otimes_{\Lambda_\mu}\C_\omega\,.
\]

Let us now deal with the tangle exterior. As~$\ell(c)=\ell(c')$, one can always obtain a colored link~$L$ from the tangle~$\tau$ by joining the punctures
with disjoint colored strands contained in the boundary of the cylinder~$D^2\times [0,1]$. By~\cite[Lemma 1]{CimConway}, it is possible to find a~$C$-complex~$S$ for the
colored link~$L$, which can be assumed to be contained in the cylinder.
By~\cite[Section 3]{CF}, it is possible to recover the free abelian covering of the link exterior by cutting it along~$S$. Consequently, if we denote by~$S_1,\dots,S_\mu$ the
components of~$S$, let~$N_i=S_i\times [-1,1]$ be a bicollar neighborhood of~$S_i\subset D^2\times [0,1]$, and
set~$N=\bigcup_{i=1}^\mu (N_i\cap X_\tau)$,~$Y=X_\tau\setminus\bigcup_{i=1}^\mu \text{Int}(N_i)$, and~$R=N \cap Y$, we can then follow the exact same steps as above. 
\end{proof}

Proposition~\ref{prop:form} now follows readily.

\begin{proof}[Proof of Proposition~\ref{prop:form}]
The isomorphism~$\psi_c$ is given by Lemma~\ref{lemma: coverings}. Note that this isomorphism is natural, in the sense that it is given by the composition of several
inclusion induced isomorphisms. In particular, it preserves the intersection numbers, so Proposition~\ref{prop:form} follows from Corollary~\ref{cor:eval}
applied to~$M=\widehat{D}_c$. (One needs to work over the ring~$\Lambda_S$ to ensure that~$H_1(\widetilde{D}_c)_S$
is free, but this is not an issue, as the homomorphism~$\Lambda_\mu\to\C$ mapping~$t_i$ to~$\omega_i\neq 1$ factors through~$\Lambda_S$.)
\end{proof}

The proof of Theorem~\ref{thm:relating} will rely on one last intermediate statement.

\begin{lemma}
\label{lemma: Relating}
Assume that~$\omega$ is in~$\mathbb{T}^\mu_P$ and let~$\tau\in T_\mu^\omega(c,c')$ be such that
the~$\Lambda_\mu$-module~$\mathit{Ker}(j_\tau)$ is Lagrangian, and its localization is a free~$\Lambda_S$-module.
Then, via the isomorphisms of Lemma~\ref{lemma: coverings}, we have
\[
\mathcal{F}_\omega(\tau)=\mathit{Ker}(j_\tau)\otimes_{\Lambda_\mu}\C_{\omega}\,.
\]
\end{lemma}

\begin{proof}
By definition, the isomorphims of Lemma~\ref{lemma: coverings} allow us to identify~$\mathcal{F}_\omega(\tau)$ with the kernel of the map 
\[
j_\tau \otimes \mathit{id}_{\C_{\omega}}\colon (H_1(\widetilde{D}_c) \oplus H_1(\widetilde{D}_c)) \otimes_{\Lambda_\mu} \C_{\omega} \rightarrow H_1(\widetilde{X}_\tau)\otimes_{\Lambda_\mu} \C_{\omega}\,.
\]
To prove the assertion, it is therefore enough to show the equality
\[
\mathit{Ker}(j_\tau \otimes \mathit{id}_{\C_{\omega}})=\mathit{Ker}(j_\tau) \otimes_{\Lambda_\mu} \C_{\omega}\,.
\]
The inclusion from right to left is straightforward. To get equality, we will argue that both spaces have the same (complex) dimension.

Assume that~$D_c$ and~$D_{c'}$ are punctured~$n$ and~$n'$ times, respectively. Then, we know that~$H_1(\widetilde{D}_c)$ and~$H_1(\widetilde{D}_{c'})$ are~$\Lambda_\mu$-modules of respective rank~$(n-1)$ and~$(n'-1)$. Since~$\omega$ belongs to~$\mathbb{T}^\mu_P$ and~$\tau$ to~$T_\mu^\omega(c,c')$, the subspace~$\mathcal{F}_\omega(\tau)$ is
Lagrangian by Proposition~\ref{prop:FLagr}. As the form on~$H_1(\widehat{D}_{c})_{\omega} \oplus H_1(\widehat{D}_{c'})_{\omega}$ is non-degenerate (Lemma~\ref{lemma: boundary}), the dimension of~$\mathit{Ker}(j_\tau \otimes \mathit{id}_{\C_{\omega}}) \cong \mathcal{F}_\omega(\tau)$ is half that of~$(H_1(\widetilde{D}_c)\otimes_{\Lambda_\mu} \C_{\omega}) \oplus (H_1(\widetilde{D}_c) \otimes_{\Lambda_\mu} \C_{\omega})$, that is,
\[
\mathit{dim}(\mathit{Ker}(j_\tau \otimes \mathit{id}_{\C_{\omega}}))=((n-1)+(n'-1))/2\,.
\]
On the other hand,~$\mathit{Ker}(j_\tau)$ is a Lagrangian submodule of a non-degenerate Hermitian~$\Lambda_\mu$-module, and its localization is free over~$\Lambda_S$;
hence, the dimension of~$\mathit{Ker}(j_\tau)\otimes_{\Lambda_\mu} \C_{\omega}=\mathit{Ker}(j_\tau)_S\otimes_{\Lambda_S}\C_{\omega}$ is also equal to~$((n-1)+(n'-1))/2$.
\end{proof}

We can finally prove Theorem~\ref{thm:relating}.

\begin{proof}[Proof of Theorem~\ref{thm:relating}:]
By standard properties of the tensor product (recall Lemma~\ref{lem: eval}), the componentwise evaluation by~$\chi_\omega$ of a matrix for the
inclusion map~$i$ of~$\mathit{Ker}(j_\tau)_S$ inside~$H_1(\widetilde{D}_c)_S \oplus H_1(\widetilde{D}_{c'})_S$ yields a matrix for
\[
i\otimes \mathit{id}_{\C_\omega}\colon\mathit{Ker}(j_\tau)\otimes_{\Lambda_\mu}\C_{\omega}\to(H_1(\widetilde{D}_c)\oplus H_1(\widetilde{D}_{c'}))\otimes_{\Lambda_\mu}\C_\omega\,.
\]
By Lemma~\ref{lemma: Relating}, this map can be identified with the inclusion of~$\mathcal{F}_\omega(\tau)$
into~$H_1(\widehat{D}_c)_{\omega} \oplus H_1(\widehat{D}_c)_{\omega}$, and the proof is completed.
\end{proof}

\bibliographystyle{plain}

\nocite{*}

\bibliography{bibliographie}

\def\cprime{$'$}
\begin{thebibliography}{10}

\bibitem{Abd}
Mohammad~N. Abdulrahim.
\newblock A faithfulness criterion for the {G}assner representation of the pure
  braid group.
\newblock {\em Proc. Amer. Math. Soc.}, 125(5):1249--1257, 1997.

\bibitem{Bir}
Joan~S. Birman.
\newblock {\em Braids, links, and mapping class groups}.
\newblock Princeton University Press, Princeton, N.J.; University of Tokyo
  Press, Tokyo, 1974.
\newblock Annals of Mathematics Studies, No. 82.

\bibitem{Bur}
Werner Burau.
\newblock \"{U}ber {Z}opfgruppen und gleichsinnig verdrillte {V}erkettungen.
\newblock {\em Abh. Math. Sem. Univ. Hamburg}, 11(1):179--186, 1935.

\bibitem{CimConway}
David Cimasoni.
\newblock A geometric construction of the {C}onway potential function.
\newblock {\em Comment. Math. Helv.}, 79(1):124--146, 2004.

\bibitem{CF}
David Cimasoni and Vincent Florens.
\newblock Generalized {S}eifert surfaces and signatures of colored links.
\newblock {\em Trans. Amer. Math. Soc.}, 360(3):1223--1264 (electronic), 2008.

\bibitem{CT}
David Cimasoni and Vladimir Turaev.
\newblock A {L}agrangian representation of tangles.
\newblock {\em Topology}, 44(4):747--767, 2005.

\bibitem{CT2}
David Cimasoni and Vladimir Turaev.
\newblock A {L}agrangian representation of tangles. {II}.
\newblock {\em Fund. Math.}, 190:11--27, 2006.

\bibitem{Cooper}
D.~Cooper.
\newblock The universal abelian cover of a link.
\newblock In {\em Low-dimensional topology ({B}angor, 1979)}, volume~48 of {\em
  London Math. Soc. Lecture Note Ser.}, pages 51--66. Cambridge Univ. Press,
  Cambridge-New York, 1982.

\bibitem{DFL}
Alex Degtyarev, Vincent Florens, and Ana Lecuona.
\newblock {The signature of a splice}.
\newblock {\em ArXiv e-prints}, 2014.

\bibitem{Flo}
Vincent Florens.
\newblock Signatures of colored links with application to real algebraic
  curves.
\newblock {\em J. Knot Theory Ramifications}, 14(7):883--918, 2005.

\bibitem{GG}
Jean-Marc Gambaudo and {\'E}tienne Ghys.
\newblock Braids and signatures.
\newblock {\em Bull. Soc. Math. France}, 133(4):541--579, 2005.

\bibitem{Ensaios}
\'Etienne Ghys and Andrew Ranicki, editors.
\newblock {\em Six papers on signatures, braids and Seifert surfaces},
  volume~30 of {\em Ensaios Matem\'aticos [Mathematical Surveys]}.
\newblock Sociedade Brasileira de Matem\'atica, Rio de Janeiro, 2016.

\bibitem{Gil}
Patrick~M. Gilmer.
\newblock Configurations of surfaces in {$4$}-manifolds.
\newblock {\em Trans. Amer. Math. Soc.}, 264(2):353--380, 1981.

\bibitem{KLW}
Paul Kirk, Charles Livingston, and Zhenghan Wang.
\newblock The {G}assner representation for string links.
\newblock {\em Commun. Contemp. Math.}, 3(1):87--136, 2001.

\bibitem{LeD}
J.-Y. Le~Dimet.
\newblock Enlacements d'intervalles et repr\'esentation de {G}assner.
\newblock {\em Comment. Math. Helv.}, 67(2):306--315, 1992.

\bibitem{L-L}
Sang~Jin Lee and Eonkyung Lee.
\newblock Potential weaknesses of the commutator key agreement protocol based
  on braid groups.
\newblock In {\em Advances in cryptology---{EUROCRYPT} 2002 ({A}msterdam)},
  volume 2332 of {\em Lecture Notes in Comput. Sci.}, pages 14--28. Springer,
  Berlin, 2002.

\bibitem{Lev}
J.~Levine.
\newblock Knot cobordism groups in codimension two.
\newblock {\em Comment. Math. Helv.}, 44:229--244, 1969.

\bibitem{MeyPhd}
Werner Meyer.
\newblock Die {S}ignatur von lokalen {K}oeffizientensystemen und
  {F}aserb\"undeln.
\newblock {\em Bonn. Math. Schr.}, (53):viii+59, 1972.

\bibitem{Mey}
Werner Meyer.
\newblock Die {S}ignatur von {F}l\"achenb\"undeln.
\newblock {\em Math. Ann.}, 201:239--264, 1973.

\bibitem{Morton}
H.~R. Morton.
\newblock The multivariable {A}lexander polynomial for a closed braid.
\newblock In {\em Low-dimensional topology ({F}unchal, 1998)}, volume 233 of
  {\em Contemp. Math.}, pages 167--172. Amer. Math. Soc., Providence, RI, 1999.

\bibitem{Mur}
Kunio Murasugi.
\newblock On a certain numerical invariant of link types.
\newblock {\em Trans. Amer. Math. Soc.}, 117:387--422, 1965.

\bibitem{Py}
Pierre Py.
\newblock Indice de {M}aslov et th\'eor\`eme de {N}ovikov-{W}all.
\newblock {\em Bol. Soc. Mat. Mexicana (3)}, 11(2):303--331, 2005.

\bibitem{Squier}
Craig~C. Squier.
\newblock The {B}urau representation is unitary.
\newblock {\em Proc. Amer. Math. Soc.}, 90(2):199--202, 1984.

\bibitem{Tri}
A.~G. Tristram.
\newblock Some cobordism invariants for links.
\newblock {\em Proc. Cambridge Philos. Soc.}, 66:251--264, 1969.

\bibitem{Turaev}
Vladimir~G. Turaev.
\newblock {\em Quantum invariants of knots and 3-manifolds}, volume~18 of {\em
  de Gruyter Studies in Mathematics}.
\newblock Walter de Gruyter \& Co., Berlin, revised edition, 2010.

\bibitem{CTC}
C.~T.~C. Wall.
\newblock Non-additivity of the signature.
\newblock {\em Invent. Math.}, 7:269--274, 1969.

\end{thebibliography}

\end{document}